\documentclass[12pt]{article}
\title{Operator level hard-to-soft transition for $\beta$-ensembles}
\author{Laure Dumaz, Yun Li and Benedek Valk\'o}

\newcommand{\Addresses}{{  
\bigskip
  \footnotesize

Laure Dumaz, \textsc{Ceremade, CNRS, UMR 7534, Universit\'e Paris-Dauphine, PSL University, 75016 Paris, France}\par\nopagebreak
  \textit{E-mail address:} \texttt{dumaz@ceremade.dauphine.fr}

  \medskip

 Yun Li, \textsc{Department of Mathematics, University of Wisconsin Madison, 480 Lincoln Drive,
Madison WI 53706}\par\nopagebreak
  \textit{E-mail address:} \texttt{li724@wisc.edu}

  \medskip

Benedek Valk\'o, \textsc{Department of Mathematics, University of Wisconsin Madison, 480 Lincoln Drive,
Madison WI 53706}\par\nopagebreak
  \textit{E-mail address:} \texttt{valko@math.wisc.edu}

}}

\usepackage{units}
\usepackage{enumitem}

\usepackage{graphicx}
\usepackage{amsmath}
\usepackage{amsthm}

\usepackage{amsfonts}
\usepackage{color}
\usepackage{hyperref}
    \newtheorem{theorem}{Theorem}
    \newtheorem{lemma}[theorem]{Lemma}
    \newtheorem{proposition}[theorem]{Proposition}
    \newtheorem{corollary}[theorem]{Corollary}

\theoremstyle{definition} 
    
    \newtheorem{remark}[theorem]{Remark}

\newcommand{\eps}{\varepsilon}
\newcommand{\Z}{{\mathbb Z}}

\newcommand{\R}{{\mathbb R}}
\newcommand{\CC}{{\mathbb C}}
\newcommand{\DD}{{\mathcal D}}

\newcommand{\N}{{\mathbb N}}

\newcommand{\lstar}{{\raise-0.15ex\hbox{$\scriptstyle \ast$}}}

\theoremstyle{remark} 

\newcommand{\Airyb}{\operatorname{Airy}_{\beta}}
\newcommand{\Bessel}{\operatorname{Bessel}}

\newcommand{\Besselop}{\textsf{G}_{\beta,2a}}

\newcommand{\Airyop}{\textsf{A}_{\beta}}

\definecolor{violet}{rgb}{0.8,0,0.2}

\newcommand{\ind}{\mathbf 1}
\newcommand{\ddd}{\mathbf{d}}
\newcommand{\psid}{\psi_{\ddd}}

\newcommand{\tph}{\tilde \phi}

\newcommand{\bside}{\noindent\textbf{\benedek{Begin side computation.}}
\begin{footnotesize}}

\newcommand{\eside}{\end{footnotesize}
\noindent \textbf{\benedek{End side computation.}}}

\begin{document}
\maketitle

\begin{abstract}
The soft and hard edge scaling limits of $\beta$-ensembles can be characterized as the spectra of certain random Sturm-Liouville operators \cite{RR, RRV}. 
It has been shown that by tuning the parameter of the hard edge process one can obtain the soft edge process as a scaling limit \cite{BF, RR, RR_spike}. 
We prove that this limit can be realized on the level of the corresponding random operators. More precisely, the random operators can be coupled in a way so that the scaled versions of the hard edge operators converge to the soft edge operator a.s.~in the norm resolvent sense. 
\end{abstract}

\section{Introduction}

The size $n$ Laguerre $\beta$-ensemble is a two-parameter family of distributions on $\R_+^n$ with density function 
\begin{align}\label{Lag_pdf}
p_{n, \beta, a}(\lambda_1, \dots, \lambda_n)=\frac{1}{Z_{n,\beta, a} }\prod_{j<k} |\lambda_j-\lambda_k|^{\beta} \prod_{k=1}^n \lambda_k^{\frac{\beta}{2}(a+1)-1} e^{- \frac{\beta}{2} \lambda_k}\,.
\end{align}
The parameters satisfy $\beta>0$ and $a>-1$, and  $Z_{n,\beta, a}$ is an explicitly computable normalizing constant. This density corresponds to the Gibbs measure of $n$ positively charged particles living on the positive half-line with a log-Gamma potential. For $\beta=1, 2$ or $4$ and $a\in \Z_{\ge 0}$, the density (\ref{Lag_pdf}) is also the joint eigenvalue distribution for an $n\times n$ Wishart matrix with real, complex or quaternion ingredients, respectively. These are classical random matrix ensembles of the form $MM^\dag$ where $M$ is an $n\times (n+a)$ dimensional matrix with i.i.d.~standard real/complex/quaternion gaussian entries. Notice that the matrix $(n+a)^{-1}\,MM^\dag$ is the correlation matrix of $n$ independent individuals whose $n+a$ characteristics are i.i.d. standard Gaussians.

When $n+a$ is of the same order as $n$, the macroscopic behavior of this ensemble is described by the famous Marchenko-Pastur limit law. Fix $\beta >0$ and let $a_n>-1, n\ge 1$ be a sequence such that $\lim_{n\to \infty} \frac{n+a_n}{n}=\gamma\in [1,\infty)$ exists. Denote by $\Lambda_{n, \beta, a_n}=(\lambda_{1,n}, \dots, \lambda_{n,n})$ a size $n$ Laguerre $\beta$-ensemble with parameter $a_n$, and consider the scaled empirical spectral measure $\nu_n:=\frac{1}{n} \sum_{k=1}^n \delta_{\lambda_{k,n}/n}$. The Marchenko-Pastur theorem (\cite{MP}, \cite{ForBook}) states that the sequence of random probability measures $\nu_n, n\ge 1$ converges in distribution a.s.~to a deterministic measure with density given by
\begin{align}\label{MP_dist}
\sigma_\gamma(x)=\frac{\sqrt{(x-b_{-})(b_+-x)}}{2\pi x} \ind_{[b_{-},b_{+}]}(x), \qquad b_{\pm}=b_{\pm}(\gamma)=(\sqrt{\gamma}\pm 1)^2.
\end{align}
Note that in the case $\gamma=1$, the density becomes $\frac{\sqrt{x(4-x)}}{2\pi x} \ind_{[0,4]}(x)$. 

The microscopic behavior of the Laguerre ensemble can be described by the large $n$ limit of the point process $c_n(\Lambda_{n,\beta, a_n}-d_n)$ where $d_n$ is the centering point and $c_n$ is the appropriate scaling parameter.
In order to get a meaningful point process limit, the scaling parameter $c_n$ would need to be chosen so that it is 
  roughly the inverse of the average spacing between the particles near $d_n$. 
From now on, we will focus on the \emph{lower edge behavior} i.e. the case $d_n := b_-$. 
(See  \cite{JV}  and \cite{RRV} for the bulk and upper edge behavior.)

The distribution of the limiting point process  depends on the asymptotic behavior of the sequence $a_n$.
If $a_n=a>-1$ does not depend on $n$, then Ram\'irez and Rider \cite{RR} showed that the scaling limit of $n \Lambda_{n,\beta,a}$ exists, and gave an explicit description of the limiting point process. This is called the hard edge  scaling limit. 
\begin{theorem}[Hard edge limit of the Laguerre ensemble, \cite{RR}]\label{thm:hardedge}
Fix $\beta>0$ and $a>-1$, and let $ \Lambda_{n,\beta, a}$  be a size $n$ Laguerre $\beta$-ensemble with  parameter $a$. Then the sequence $n \Lambda_{n,\beta, a}$ converges in distribution to a point process $\Bessel_{\beta,a}$ as $n\to \infty$. The $\Bessel_{\beta,a}$ process has the same distribution as the a.s.~discrete spectrum of the random differential operator
\begin{align}\label{hardedgeop}
&\qquad \mathfrak{G}_{\beta,a}=-\frac{1}{m(x)} \frac{d}{dx}\left(\frac{1}{s(x)} \frac{d}{dx}\, \cdot\,  \right)\,,\\[5pt]
\label{ms}
m(x)=m_{a}(x)&=e^{-(a+1)x-\frac{2}{\sqrt{\beta}} B_a(x)}\,, \qquad s(x)=s_a(x)=e^{ ax +\frac{2}{\sqrt{\beta}} B_a(x)}.
\end{align} 
Here $B_a$ is a standard Brownian motion, and the operator $\mathfrak{G}_{\beta,a}$ is 
defined on a subset of $L^2(\R_+,m)$ with Dirichlet boundary condition at $0$ and Neumann at infinity.
\end{theorem}
We will come back to the precise definition of $\mathfrak{G}_{\beta,a}$ in Section \ref{s:SL}. Let us just mention that since the functions $s, m$ are a.s.~continuous, this differential operator fits into the  framework of classical Sturm-Liouville operators. 

If the sequence $a_n, n\ge1$ goes to infinity with at least a constant speed then the Marchenko-Pastur theorem and the expression of the limiting measure (\ref{MP_dist}) suggest a different scaling than the one seen in the hard edge case. 
%
%
%
%
This is called the soft edge scaling scaling limit. The description of the limiting point process follows from the work of \cite{RRV}.\begin{theorem}[Soft edge limit, \cite{RRV}]\label{thm:softedge}
Fix $\beta>0$ and suppose that the sequence $a_n, n\ge1$ satisfies $\liminf_{n\to \infty} a_n/n>0$. Then there is a point process $\Airyb$ so that the following limit in distribution holds as $n\to \infty$:
\[
\frac{((n+a_n) n)^{1/6}}{(\sqrt{n+a_n}-\sqrt{n})^{4/3}}(\Lambda_{n,\beta,a_n}-(\sqrt{n+a_n}-\sqrt{n})^2)\Rightarrow \Airyb.
\]
The point process $\Airyb$ has the same distribution as the a.s.~discrete spectrum of the random differential operator 
\begin{align}\label{Airyop}
\Airyop=-\frac{d^2}{dx^2}+x+\frac{2}{\sqrt{\beta}} B'
\end{align}
defined on a subset of $L^2(\R_+)$ with Dirichlet boundary conditions at 0. Here  $B'$ is the standard white noise on $\R_+$.
\end{theorem}
The precise definition of the operator $\Airyop$ will be discussed in Section \ref{s:SL}. Note that a priori it is not even clear that the operator $\Airyop$ is well-defined, due to the irregularity of the white noise term in the potential. 

It is natural to conjecture that the condition $\liminf_{n\to \infty} a_n/n>0$ in Theorem \ref{thm:softedge} could be relaxed to $\lim_{n\to \infty} a_n=\infty$, but the tools developed in \cite{RRV} do not seem to be sufficient to prove this. (See however \cite{DMT} for the treatment of the case  $\beta=2$, $a_n=c\sqrt{n}$, where the appropriate limit is proved using the determinantal structure present at $\beta=2$.) This conjecture, together with a diagonal argument, would imply the following point process level transition from the $\Bessel_{\beta,a}$ process to $\Airyb$:
\begin{align}\label{HtoS_process}
a^{-4/3}(\Bessel_{\beta,2a}-a^2)\Rightarrow \Airyb, \qquad \text{as $a\to \infty$}.
\end{align}
See \cite{BVBV} for a similar diagonal argument for the transition between the soft edge and the bulk limiting processes.

The process level limit  (\ref{HtoS_process}) is called hard to soft edge transition. It can be analyzed without considering the finite $n$ ensembles, working directly with the limiting point processes appearing in the statement. This transition was first proved in \cite{BF} for $\beta=2$ using again the determinantal structure present in this case. For general $\beta>0$, Ram\'irez and Rider \cite{RR} proved the scaling limit for the first point of the respective point processes. This result was extended in \cite{RR_spike} to a full process level limit.

In light of Theorems \ref{thm:hardedge} and \ref{thm:softedge}, the statement of (\ref{HtoS_process}) can be rewritten using the operators $\mathfrak{G}_{\beta,2a}$ and $\Airyop$ as 
 \[
 a^{-4/3}(\operatorname{spec} (\mathfrak{G}_{\beta,2a} )-a^2)\Rightarrow \operatorname{spec} (\Airyop),\]
 where $\operatorname{spec} (Q)$ denotes the spectrum of the operator $Q$. 
It is natural to ask whether it is possible to prove the corresponding limit on the level of the operators. This is the main result of our paper. Theorem \ref{thm:hardtosoft} below shows that one can realize the operator level limit as an a.s.~limit with an appropriate coupling between the Brownian motion $B_a$ of the Bessel operator \eqref{hardedgeop} and the white noise $B'$ of the Airy operator \eqref{Airyop}.

To describe our coupling, we introduce a simple transformation of $\mathfrak{G}_{\beta,2a}$. 
For $a>0$ let $\theta_a$ be the `stretching' transformation defined via
\begin{align}\label{stretch}
\left(\theta_a f\right)(x)=f(a^{2/3} x),
\end{align}
and define the following transform of the hard-edge operator corresponding to $2a$: 
\begin{align}\label{Besselop}
\Besselop= \theta_a^{-1} \left(m_{2a}^{1/2} \mathfrak{G}_{\beta,2a} m_{2a}^{-1/2}\right) \theta_a.
\end{align}
As we will see in Section \ref{s:SL}, $\Besselop$ is a self-adjoint operator with the same spectrum as $\mathfrak{G}_{\beta,2a} $, and the operators $\Airyop^{-1}$ and $(\Besselop-a^2)^{-1}$ are  Hilbert-Schmidt integral operators acting on the same space of $L^2(\R_+)$ functions. 

\begin{theorem}[Operator level hard-to-soft transition]\label{thm:hardtosoft}
Let $B'$ be  white noise on $\R_+$ and let $B$ be a Brownian motion defined as $B(x) := \int_0^x B'(y)dy$. Set $B_{2a}(x)=a^{-1/3} B(a^{2/3} x)$ for $a>0$. Consider $\Airyop$ defined as (\ref{Airyop}) using the white noise $B'$, and $\Besselop$ defined with the Brownian motion $B_{2a}$ via (\ref{hardedgeop}) and (\ref{Besselop}) for $a>0$. Then $a^{4/3}(\Besselop-a^2)^{-1}\to \Airyop^{-1}$ a.s.~in Hilbert-Schmidt norm as $a\to \infty$.
\end{theorem}
We expect that with a more careful application of our methods one could also get estimates on the speed of convergence in our coupling. See Remark \ref{rmk:convergence rate} in Section \ref{s:Step3proof}.

The theorem implies that $a^{-4/3}(\Besselop-a^2)\to \Airyop$ a.s.~in norm resolvent sense from which the process level transition  $a^{-4/3}(\operatorname{spec} (\mathfrak{G}_{\beta,2a} )-a^2)\Rightarrow \operatorname{spec} (\Airyop)$, and therefore the limit (\ref{HtoS_process}) follow. The coupling of the operators produces a coupling of the point processes in a way that almost surely the points in the scaled hard edge processes converge to the points in the soft edge point process. More precisely, a version of the Hoffman-Wielandt inequality (see e.g.~\cite{BhatiaElsner})  shows that if we denote the ordered points in the scaled hard edge process $a^{-4/3}(\Bessel_{\beta,2a}-a^2)$ by $\lambda_{k,2a}, k\ge 0$, and the ones in the soft edge process $\Airyb$ by $\lambda_{k}, k\ge 0$, then in the coupling of Theorem \ref{thm:hardtosoft} we have a.s.
\[
\lim_{a\to \infty} \sum_{k=0}^\infty \left|\lambda_k^{-1}-\lambda_{k,2a}^{-1}  \right|^2=0\;.
\]
Moreover, as the spectrum of the operators are discrete, and each eigenvalue has multiplicity $1$, the a.s.~norm resolvent convergence also implies the a.s.~convergence of the respective normalized eigenfunctions in $L^2$.

\medskip

The structure of the rest of the paper is as follows. 
In Section \ref{s:SL}  we show how one can describe the appearing differential operators using the generalized Sturm-Liouville theory, show that $\Airyop^{-1}$ and $(\Besselop-a^2)^{-1}$ are Hilbert-Schmidt integral operators, and  describe their kernels in terms of certain diffusions. Section \ref{s:mainsteps} outlines the main steps of the proof of the main Theorem \ref{thm:hardtosoft}. Our proof uses the approximation of the integral operators by their truncated version. We state the convergences of the truncated operators towards their full operator as well as the convergence of the truncated hard edge integral operators to the truncated soft edge integral operator in several lemmas whose proofs are postponed to later sections. Section \ref{s:Step1proof} estimates the truncation error of the soft edge integral operator. Section \ref{s:Step2proof} shows that the truncated hard edge integral operators converge to the truncated soft edge integral operator by proving that the integral kernels converge uniformly on compacts with probability one. Section \ref{s:Step3proof} describes the asymptotic behavior of the diffusions connected to the operator $ \Besselop$ and provides the results needed to estimate the truncation error for the hard edge integral operators. Finally, the final section gathers the proof of some technical lemmas needed for the results of Sections \ref{s:Step1proof} and \ref{s:Step3proof}. 

%
%
\section{The operators $\Airyop$ and $\mathfrak{G}_{\beta,2a}$ as generalized Sturm-Liouville operators}\label{s:SL}

This section briefly introduces the background for the differential operators appearing in this work, and shows how it can be used to describe the random differential operators $ \mathfrak{G}_{\beta,2a}, \Besselop, \Airyop$ and their inverses. We use the classical theory discussed in \cite{Weidmann} and Chapter 9 of \cite{Teschl}. 

\subsection{Generalized Sturm-Liouville operators}

We consider generalized Sturm-Liouville (S-L) operators of the form
 \begin{align}\label{SL}
\tau u(x)= \frac{1}{r(x)}\left(-(p_1(x)u'(x)-q_0(x) u(x))'-q_0(x)u'(x)+p_0(x) u(x)\right),
 \end{align}
 where $u$ is a real valued function on $[0,L]$ for some $L >0$ or on $\R_+$ (which we consider to be the $L=\infty$ case in the following). We assume that the real functions $p_0, p_1, q_0, r$  are continuous on $[0,\infty)$ and $r(x), p_1(x)>0$ for $x\ge 0$.

The operation $\tau u$ is well-defined if both $u$ and $p_1u'-q_0 u$ are absolutely continuous on $[0,L]$.  From the standard theory of differential equations we have that for any $\lambda\in \CC$ the differential equation $\tau u=\lambda u$ has a unique differentiable solution on $[0,L]$ with initial conditions $u(0)=c_0, u'(0)=c_1$. We  note that  if $f_1, f_2$ are both solutions of $\tau f=\lambda f$ then integration by parts shows that the Wronskian $p_1 (f_1 f_2'-f_1' f_2)$ is constant on $\R_+$.

We consider differential operators satisfying the following three assumptions: 
\begin{enumerate}
\item[(A1)] The solution $u_{\ddd}$ of the equation $\tau u_{\ddd}=0$ with Dirichlet initial condition $u_{\ddd}(0)=0$, $u_{\ddd}'(0)=1$ is not in $L^2(\R_+, r)$, i.e.~$\int_0^\infty u_{\ddd}^2(x) r(x) dx=\infty$.

\item[(A2)] There is a unique solution $u_\infty$ of the equation $\tau u_\infty=0$, with initial condition $u_\infty(0)=1$ that is in $L^2(\R_+,r)$. 
\item[(A3)] With $u_{\ddd}, u_\infty$ defined from (A1), (A2), we have $\int_0^\infty \int_0^x  u_\infty(x)^2 u_{\ddd}(y)^2 r(x) r(y)\, dy dx<\infty$.
\end{enumerate}

Under these assumptions, the operator $\tau$ can be made self-adjoint on an appropriate subset of $L^2([0,L],r)$ or $L^2(\R_+,r)$. We introduce 
\begin{align*}
\DD_L=\left\{u\in L^2([0,L],r):  \tau u\in L^2([0,L],r), \, u, p_1u'-q_0 u\in \text{AC}([0,L])\right\},
\end{align*}
and we drop the subscript $L$ for $L=\infty$. Here $\text{AC}([0,L])$ is the set of absolutely continuous real functions on $[0,L]$. 

The continuity of the functions $p_0$, $p_1$, $q_0$ and $r$ implies that the operator $\tau$ is regular at $0$ and at any finite $L$ and therefore is limit circle at those points.
The condition (A1) implies that the operator $\tau$ is limit point at $+\infty$ thanks to the Weyl's alternative theorem. Conditions (A2) and (A3) ensure that the inverse and the resolvent are Hilbert Schmidt operators.

The following propositions summarize the basic properties of generalized Sturm-Liouville differential operators satisfying conditions (A1)-(A3). 

\begin{proposition}[Self-adjoint version of $\tau$]\label{pr:SL1}
Assume that $\tau$ is of the form (\ref{SL}) and that it satisfies the condition (A1-A3), 
and let $L\in (0,\infty]$. Then there is a self-adjoint version of the operator on $[0,L]$ with Dirichlet boundary conditions on the domain 
\[
\DD_{L,0}=\DD_L\cap \{u: u(0)=0, u(L)=0\},
\]
 where the end condition $u(L)=0$ is dropped in the case $L=\infty$. We denote this self-adjoint operator by $\tau_L$. 
\end{proposition} 
 
 \begin{proposition}[Inverse as an integral operator]\label{pr:SL2}
Consider the operator $\tau_L$ from   Proposition \ref{pr:SL1}. If $L$ is finite then assume  that $u_{\ddd}(L)\neq 0$ (i.e.~that $0$ is not an eigenvalue of $\tau_L$). 
Then the  inverse $\tau_L^{-1}$ is an integral operator of the  form $\tau_L^{-1} f(x)=\int_0^L K^{(L)}(x,y) f(y) r(y) dy$ on $L^2([0,L],r)$ with
\begin{align}\label{HS_K}
K^{(L)}(x,y)=\frac{1}{p_1(0)} \left(u_L(x) u_{\ddd}(y) \ind(x\ge y)+u_{\ddd}(x) u_L(y) \ind(x<y)    \right).
\end{align}
Here $u_{\ddd}$ is defined in (A1). If $L=\infty$ then $u_L$ is $u_\infty$ from (A2), and in the case $L<\infty$ the function $u_L$ is defined as 
the solution of $\tau u_L=0$ with $u_L(0)=1, u_L(L)=0$. 
The inverse operator $\tau_L^{-1}$ is a Hilbert-Schmidt operator in $L^2([0,L],r)$, and it has a bounded pure point spectrum.  
\end{proposition}

\begin{proposition}[Resolvent as an integral operator]\label{pr:SL3}
Consider $\tau_L$ from Proposition \ref{pr:SL1}, and assume that a given $\lambda\in \R$  is not an eigenvalue of $\tau_L$. Then the resolvent $(\tau_L-\lambda)^{-1}$ is  a Hilbert-Schmidt integral operator of the same form as $K^{(L)}$ from (\ref{HS_K}), where now $u_{\ddd}, u_L$ are the appropriate solutions of $\tau u=\lambda u$ with the respective boundary conditions. For $L=\infty$ the function $u_L=u_\infty$ is the unique solution of $\tau u_\infty=\lambda u_\infty$ with $u_\infty(0)=1$ and $u_\infty\in L^2(\R_+,r)$.
\end{proposition}
The proofs of these propositions follow from the theory of Sturm-Liouville operators. Again, we refer to the monograph \cite{Weidmann}. Note that the classical theory (when $q_0=0$) is treated in a self-contained way in Chapter 9 of  \cite{Teschl} (see in particular Theorems 9.6 and 9.7).

\subsection{Bessel and Airy operators as generalized S-L operators}

The operators $\mathfrak{G}_{\beta,a}$, $\Besselop$, and $\Airyop$ can be represented as a generalized Sturm-Liouville operators for which Assumptions (A1-A3) are satisfied, and hence the appropriate resolvents are a.s.~Hilbert-Schmidt integral operators. We summarize the relevant results in the propositions below. 

\begin{proposition}[$\mathfrak{G}_{\beta,2a}$ as a Sturm-Liouville operator]\label{pr:hardSL}
The operator $\mathfrak{G}_{\beta,2a}$ is a Sturm-Liouville operator of the form (\ref{SL}) with $r=m_{2a}, p_1=s_{2a}^{-1}$, $p_0=q_0=0$. The operator satisfies the conditions (A1-A3) with probability one if $a>1/2$. 

If $\phi$ solves the equation $\mathfrak{G}_{\beta,2a} \phi=\lambda \phi$ with deterministic initial conditions $\phi(0)=c_0$, $\phi'(0)=c_1$ then $(\phi, \phi')$ is the unique strong solution of the stochastic differential equation system
\begin{align}\label{sdehard}
d\phi(x)=\phi'(x) dx, \quad d\phi'(x)=\tfrac{2}{\sqrt{\beta}}\phi'(x) dB_{2a}(x)+\left((2a+\tfrac2{\beta}) \phi'(x)-\lambda e^{-x} \phi(x)\right) dx,
\end{align}
with the corresponding initial conditions.
\end{proposition}
\begin{proof}
The fact that $\mathfrak{G}_{\beta,2a}$ is a Sturm-Liouville operator is contained in the statement of Theorem \ref{thm:hardedge}, the statement about the solution of the eigenvalue equation can be checked with It\^o's formula (see \cite{RR}). As explained in \cite{RR_err}, the Neumann boundary condition for $\mathfrak{G}_{\beta,2a}$ at $\infty$ for $a>0$ can be dropped. The SDE (\ref{sdehard}) satisfies the usual conditions for existence and uniqueness, so $(\phi,\phi')$ is a well-defined process for all times.

We only need to check that the conditions (A1-A3) are satisfied for $a>1/2$. 
This can be done directly using the a.s.~sublinear growth of the Brownian motion by noting that $u_{\ddd}(x)=\int_0^x s_{2a}(y) dy$ and $u_\infty(x)=1$. 
\end{proof}
\begin{proposition}[Integral kernel for $(\Besselop-a^2)^{-1}$]\label{prop:hardK}
For a given $a>1/2$, let $ \phi_{\ddd}^{(2a)}$ be the unique strong solution of (\ref{sdehard}) with $\lambda=a^2$ and initial conditions 
$\phi(0)=0$, $\phi'(0)=1$. Let $\mathcal E_a$ be the event  that $a^2$ is not an eigenvalue of $\mathfrak{G}_{\beta,2a}$.
Denote by $\phi_\infty^{(2a)}$  the unique solution of $\mathfrak{G}_{\beta,a} \phi_\infty^{(2a)}=a^2 \phi_\infty^{(2a)}$ with $\phi_\infty^{(2a)}(0)=1$ and  $ \phi_\infty^{(2a)}\in L^2(\R_+,m_{2a})$, this exists on $\mathcal E_a$.
 Then on the event $\mathcal E_a$  the operator $a^{4/3}(\Besselop-a^2)^{-1}$ is a Hilbert-Schmidt integral operator in $L^2(\R_+)$ with integral kernel 
\begin{align*}
K_{\textsf{G},2a}(x,y)&=\tilde \phi_\infty(x) \tilde \phi_{\ddd}(y) \ind(x\ge y)+\tilde \phi_{\ddd}(x) \tilde \phi_\infty(y) \ind(x<y),
\end{align*}
where
\begin{align}\label{phitilde}
\tilde \phi_{\ddd}(x)=a^{2/3} m_{2a}^{1/2}(a^{-2/3} x)  \phi_{\ddd}^{(2a)}(a^{-2/3} x), \qquad \tilde\phi_\infty(x)=m_{2a}^{1/2}(a^{-2/3} x)  \phi_\infty^{(2a)}(a^{-2/3} x). 
\end{align}
On the event $\mathcal{E}_a$ the operator  $a^{4/3}(\Besselop-a^2)^{-1}$ has a bounded pure point spectrum that is the same as the spectrum of $ a^{4/3} (\mathfrak{G}_{\beta,2a} -a^2)^{-1}$.
\end{proposition}
\begin{proof}
By Proposition  \ref{pr:SL3}, the function $\phi_\infty^{(2a)}$ is well-defined on $\mathcal{E}_{a}$, and  the operator $(\mathfrak{G}_{\beta,2a}-a^2)^{-1}$ is Hilbert-Schmidt on  $L^2(\R_+,m_{2a})$ with 
 integral kernel 
\[
K_{\mathfrak{G},2a}(x,y)= \phi^{(2a)}_\infty(x)  \phi^{(2a)}_{\ddd}(y) \ind(x\ge y)+ \phi^{(2a)}_{\ddd}(x)  \phi^{(2a)}_\infty(y) \ind(x<y).
\]
Recalling the definition of $\Besselop$ from (\ref{Besselop}) we get that $a^{4/3} (\Besselop-a^2)^{-1}$ is a Hilbert-Schmidt integral operator on $L^2(\R_+)$ with kernel 
\begin{align*}
K_{\textsf{G},2a}(x,y)&=a^{2/3} m_{2a}^{1/2}(a^{-2/3} x)\, K_{\mathfrak{G},2a}(a^{-2/3}x ,a^{-2/3} y) \,m_{2a}^{1/2}(a^{-2/3} y),
\end{align*}
from which the proposition follows.
\end{proof}
Note that for any fixed $a>1/2$, the event $\mathcal E_a$ has a probability $1$, see Remark \ref{rmk:qfluctuaction}. Later, in Corollary \ref{prop:eigenvalue}  in Section \ref{s:Step3proof} we show that \emph{in our coupling} if $a$ is large enough then $a^2$ is not  an eigenvalue for $\mathfrak{G}_{\beta,2a}$.

\begin{proposition}[The operator $\Airyop$ as a generalized S-L operator]\label{pr:AirySL}
The operator $\Airyop$ is a generalized Sturm-Liouville operator of the form (\ref{SL}) with $r(x)= p_1(x)=1$, $q_0(x)=\tfrac{2}{\sqrt{\beta}}B(x)$, $p_0(x)=x$. The operator  satisfies the conditions (A1-A3) with probability one. 

If $\psi$ solves the equation $\Airyop \psi=0$ with deterministic initial conditions $\psi(0)=c_0$, $\psi'(0)=c_1$, $(c_0,c_1)\neq (0,0)$, then $(\psi, \psi')$ is the strong solution of the SDE system
\begin{align}\label{AirySDE_1}
d\psi(x)=\psi'(x) dx, \qquad d\psi'(x)=\psi(x)\left(\tfrac{2}{\sqrt{\beta}} dB+x dx   \right),
\end{align}
which is well defined for all times, and satisfies
\begin{align}\label{Airysqr}
\frac{\psi'(x)}{\psi(x) \sqrt{x}}\to 1 \quad\text{a.s.~as }x\to \infty.
\end{align}
A.s.~0 is not an eigenvalue of $\Airyop$, and the operator $\Airyop^{-1}$ is a Hilbert-Schmidt integral operator with kernel
\begin{align}\label{AiryopHS}
K_{\textsf{A}}(x,y)=\psi_\infty(x) \psi_{\ddd}(y) \ind(x\ge y)+\psi_{\ddd}(x) \psi_\infty(y) \ind(x<y).
\end{align}
Here  $\psi_{\ddd}$ is the solution of $\Airyop \psi=0$ with initial condition $\psi_{\ddd}(0)=0$,  $\psi_{\ddd}'(0)=1$, and
  $\psi_{\infty}\in  L^2(\R_+)$ is the unique function satisfying 
$
\Airyop  \psi_{\infty}=0$, $\psi_\infty(0)=1$ (see Figure \ref{RT_psidpsiinfty}).
\end{proposition}

\begin{proof}
The fact that  the soft-edge operator $\Airyb$ can be represented as a generalized Sturm-Liouville operator of the form (\ref{SL}) with the listed coefficients was shown in \cite{AlexPhD} (see also \cite{Minami}). The SDE representation of the solutions of   $\Airyop \psi=0$ with a deterministic initial condition is shown in \cite{RRV}. Since the SDE (\ref{AirySDE_1}) satisfies the usual conditions of existence and uniqueness for SDEs, the solution is well defined for all times. The asymptotics (\ref{Airysqr})
was stated without proof in \cite{RRV},  we include a proof of this statement in Proposition \ref{pr:Airy} in Section \ref{s:Xlemmas} below for completeness.  

To check that the conditions (A1)-(A3) are satisfied we first observe that if $\psi_{\ddd}$ is the solution of $\Airyop \psi=0$ with Dirichlet initial condition then by (\ref{Airysqr}) for any fixed $\eps>0$ we have
\begin{align}\label{psibnd}
 e^{\left(2/3-\eps\right) x^{3/2}}\le \psi_{\ddd}(x)\le e^{(2/3+\eps) x^{3/2}}\quad \textup{ for $x$ large enough,}
 \end{align}
 hence $\psi_{\ddd}$ is not in $L^2(\R_+)$. This means that a.s.~there can be at most one $L^2(\R_+)$ solution of $\Airyop \psi=0$ with initial condition $\psi(0)=1$. We will  construct such a function using $\psi_{\ddd}$.
 
Denote by $z_0$ the largest zero of $\psi_{\ddd}$ on $\R_+$, and let $z_0=0$ if such a zero does not exists. 
 Motivated by the Wronskian identity we introduce the function 
\begin{align}\label{Airy_infty}
\psi_\infty(x)=\psi_{\ddd}(x)\int_x^\infty \psi_{\ddd}(y)^{-2} dy
\end{align}
which is well defined for $x>z_0$. One can  check that $\psi_\infty$ satisfies $ \Airyop \psi_\infty=0$ and the Wronskian identity
 \begin{align}\label{Wr_0}
\psi_\infty'(x) \psi_{\ddd}(x)-\psi_\infty(x) \psi'_\ddd(x)= -1\,
\end{align}
for $x>z_0$. Then, the function $\psi_\infty$  can be uniquely extended to $\R_+$ as a solution of $\Airyop \psi=0$. This function satisfies (\ref{Wr_0}) on $\R_+$, hence it will satisfy $\psi_\infty(0) = 1$.

Using (\ref{Airy_infty}) we see that for $x>z_0$ we have
\begin{align*}
\psi_{\ddd}(x)\psi_\infty(x)=\int_x^\infty \frac{\psi_{\ddd}(x)^2}{\psi_{\ddd}(y)^2} dy= \int_x^\infty \exp\left(-2\int_x^y \frac{\psi'_{\ddd}(z)}{\psi_{\ddd}(z)}dz\right) dy,
\end{align*}
and from  (\ref{Airysqr}) we get the bounds 
\begin{align}\label{ineq:prodpsi}
&\sqrt{y} \int_0^y \psi_{\ddd}(x)^2  \psi_{\ddd}(y)^{-2}dx\leq C,\qquad \sqrt{y} \int_y^\infty \psi_{\ddd}(x)^{-2}  \psi_{\ddd}(y)^{2}dx\leq C,
\end{align}
for some random $C<\infty$.  Together with the bound (\ref{psibnd}) this is now sufficient to show that $\psi_\infty$ is in $L^2(\R_+)$, and that 
\[\int_0^\infty \int_0^x \psi_\infty(x)^2 \psi_{\ddd}(y)^2 dy \,dx<\infty.\]
By Propositions \ref{pr:SL2}  and \ref{pr:AirySL} it follows immediately that  $\Airyop^{-1}$ is almost surely a Hilbert-Schmidt integral operator with kernel given in (\ref{AiryopHS}).
\end{proof}

\begin{figure}[!h]
\centering
\includegraphics[width = 10cm]{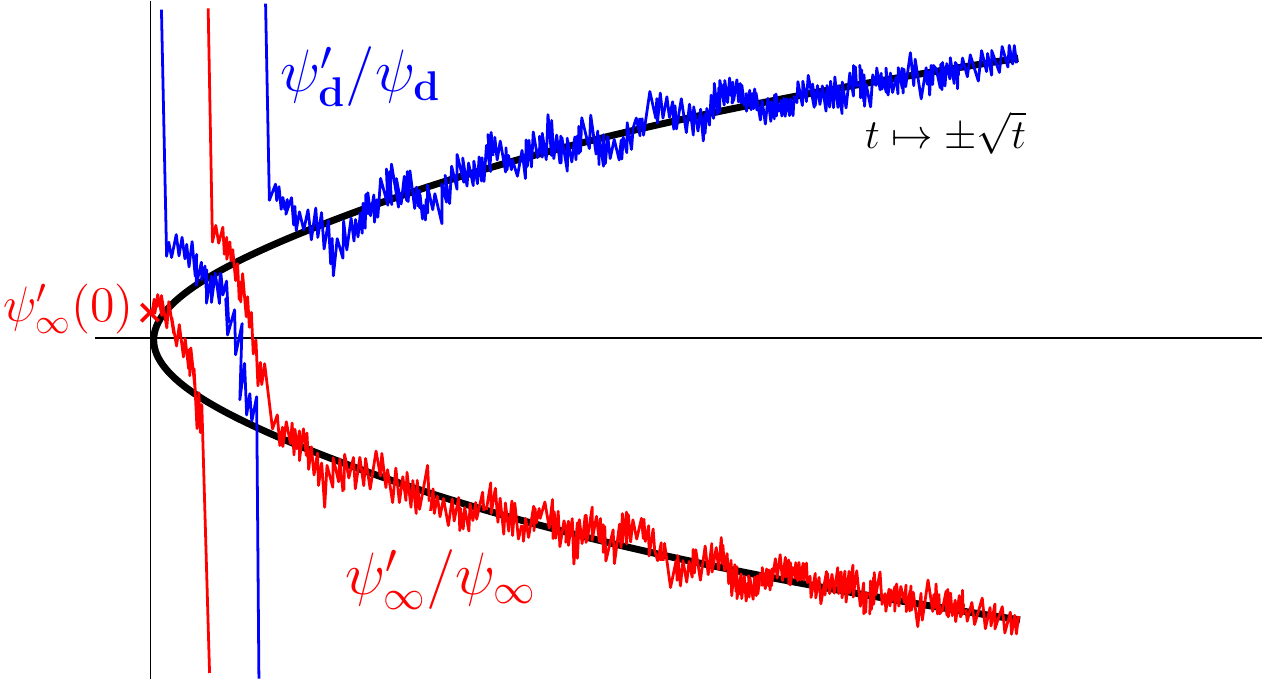}
\caption{Representation of the log-derivatives  of $\psi_\ddd$ and $\psi_\infty$.  
}\label{RT_psidpsiinfty}
\end{figure}

\begin{remark}
Using the identity (\ref{Airy_infty}) and the limit (\ref{Airysqr}) one can show that $\psi'_{\infty}(x)/\psi_\infty(x)\to -\sqrt{x}$ a.s.~as $x\to \infty$, and that $\psi_\infty(x)\le e^{-(2/3-\eps)x^{3/2}}$ for $x$ large enough. This behavior was also noted in \cite{RRV}.  See Figure \ref{RT_psidpsiinfty} for an illustration for the behavior of $\psi_{\ddd}, \psi_{\infty}$. 
\end{remark}

%
We record here the Wronskian identities for the appropriate operators: 
\begin{align}\label{Wr_1}
\psid(x) \psi_\infty'(x)-\psid'(x) \psi_\infty(x)=-1, \qquad \phi_{\ddd}(x) \phi_\infty'(x)-\phi_{\ddd}'(x) \phi_\infty(x)=-s_{2a}(x).
\end{align}
where we dropped the $a$-dependence in $\phi^{(2a)}_{\ddd}, \phi^{(2a)}_{\infty}$ to alleviate the notation.
From the second equation of (\ref{Wr_1}) one can obtain  the following analogue of the identity (\ref{Airy_infty}) for the hard edge diffusions: 
\begin{align}\label{wronski}
\phi_\infty(x)= \phi_{\ddd}(x)\int_x^\infty \phi_{\ddd}(y)^{-2} s_{2a}(y)dy,
\end{align}
if $x$ is larger than the largest zero of  $\phi_{\ddd}$.

Note that the functions $\psid, \phi_{\ddd}$ are diffusions with respect to the natural filtrations of the Brownian motions $B, B_{2a}$. This is not the case for the functions $ \psi_\infty$ and $\phi_{\infty}$, as the starting values of these processes depend on the $\sigma$-field generated by the whole Brownian motion $B(t), t\ge 0$. In particular, those functions are not Markovian.

\section{Proof of  Theorem \ref{thm:hardtosoft}}\label{s:mainsteps}

\begin{proof}[Proof of  Theorem \ref{thm:hardtosoft}]
In order to prove the theorem, we first need to show that in our coupling with probability one $a^2$ is not an eigenvalue of the operator $\Besselop$ if $a$ is large enough. This will be the content of Corollary \ref{prop:eigenvalue} in Section \ref{s:Step3proof}: we will show that there is an a.s.~finite random variable $C_{\text{ev}}$ such that the operator $\Besselop-a^2$ is invertible for all $a>C_{\text{ev}}$. In particular, this means that on the event  $\{a>C_{\text{ev}}\}$ the operator  $(\Besselop-a^2)^{-1}$ is a well-defined integral operator with kernel given in Proposition \ref{prop:hardK}.

By the results of Section \ref{s:SL}, to prove Theorem \ref{thm:hardtosoft} we need to show that  we have
\begin{align}\label{HSlim}
\lim_{a\to \infty} \int_0^\infty \int_0^\infty \left|K_{\textsf{A}}(x,y)- K_{\textsf{G},2a}(x,y) \right|^2 dx \,dy=0 \qquad  a.s. 
\end{align}
We  do this by  approximating $K_{\textsf{A}}$ and $K_{\textsf{G},2a}$ with the resolvent kernels of   the appropriate differential operators  restricted to $[0,L]$, with $L>0$. We denote these operators by $K_{\textsf{A}}^{(L)}$ and $K_{\textsf{G},2a}^{(L)}$. More specifically, set
\begin{align}\label{KAL}
K_{\textsf{A}}^{(L)}(x,y)=\psi_L(x) \psi_{\ddd}(y) \ind(y\le x\le L)+\psi_{\ddd}(x) \psi_L(y) \ind(x<y\le L),
\end{align}
where $\psi_L$ which solves $\Airyop \psi=0$ with boundary conditions $\psi_L(0)=1$, $\psi_L(L)=0$. The function $\psi_L$ is well-defined if $\psi_{\ddd}(L)\neq 0$. 

Moreover, set
\begin{align}\label{KGL}
K_{\textsf{G},2a}^{(L)}(x,y)=\tilde \phi_L(x) \tilde \phi_{\ddd}(y) \ind(y\le x\le L)+\tilde \phi_{\ddd}(x) \tilde \phi_L(y) \ind(x<y\le L)
\end{align}
where 
\[
\tilde\phi_L(x)=m_{2a}^{1/2}(a^{-2/3} x)  \phi_{a^{-2/3} L} (a^{-2/3} x),
\]
and $\phi_{a^{-2/3} L} $ solves the equation $\mathfrak{G}_{\beta,2a} \phi=a^2 \phi$ with $\phi_{a^{-2/3} L} (0)=1$, $\phi_{a^{-2/3} L} (a^{-2/3} L)=0$. 
The function $\tilde\phi_L$ is well-defined if $\phi_{\ddd}(a^{-2/3} L)\neq 0$. 
(Note that $\phi$ and $\tilde \phi$  depend on $a$ as well, which we do not denote.)

By the triangle inequality we have
\[
\|K_{\textsf{A}}-K_{\textsf{G},2a}\|_2\le \|K_{\textsf{A}}-K_{\textsf{A}}^{(L)}\|_2+\|K_{\textsf{A}}^{(L)}-K_{\textsf{G},2a}^{(L)}\|_2+\|K_{\textsf{G},2a}-K_{\textsf{G},2a}^{(L)}\|_2.
\]
We will show that all three terms on the right will vanish in the limit if we let $a\to \infty$ and then $L\to \infty$ along a particular sequence, this is the content of the Lemmas \ref{l:Step1}, \ref{l:Step2} and \ref{l:Step3} below.   From these three lemmas, we deduce the convergence (\ref{HSlim}), and hence Theorem \ref{thm:hardtosoft} follows.
\end{proof}

More precisely, we will prove the following three lemmas.
\begin{lemma}[Truncation of the Airy operator]\label{l:Step1}
$\|K_{\textsf{A}}-K_{\textsf{A}}^{(L)}\|_2^2\to 0$ a.s.~as $L\to \infty$.
\end{lemma}
\begin{lemma}[Convergence of the truncated operators]\label{l:Step2}
For any fixed $L>0$ we have
\begin{align*}
\|K_{\textsf{A}}^{(L)}-K_{\textsf{G},2a}^{(L)}\|_2^2\to 0 \mbox{ a.s.~as }a\to \infty\,.
\end{align*}
\end{lemma}
\begin{lemma}[Truncation of the Bessel operator]\label{l:Step3}
With probability $1$, we have,
\begin{align*}
\lim\limits_{L\to \infty} \limsup\limits_{a\to \infty} \|K_{\textsf{G},2a}-K_{\textsf{G},2a}^{(L)}\|_2^2=0\,.
\end{align*}
\end{lemma}

\medskip

We prove Lemma \ref{l:Step1} in Section \ref{s:Step1proof} using the  the asymptotics (\ref{psibnd}). The proof of Lemma \ref{l:Step2} is given in Section \ref{s:Step2proof}, we will  show that for  a fixed $L<\infty$ the kernel $K_{\textsf{G},2a}^{(L)}$ converges uniformly to $K_{\textsf{A}}^{(L)}$ on $[0,L]^2$ as $a\to \infty$.
 Finally, the proof of Lemma \ref{l:Step3} will be given in Section \ref{s:Step3proof}, and it will rely on a careful analysis of the asymptotic behavior of $\phi_{\ddd}^{(2a)}$. 
%

\section{Truncation of the Airy operator} \label{s:Step1proof}

We analyze the solutions of the SDE  (\ref{AirySDE_1}) via the Riccati transform  $\frac{\psi'(t)}{\psi(t)}$. Suppose that $\psi, \psi'$ is the strong solution of the SDE (\ref{AirySDE_1}) with deterministic initial conditions $\psi(0)=c_0$, $\psi'(0)=c_1$, $(c_0,c_1)\neq (0,0)$. Set $X(t)=\frac{\psi'(t)}{\psi(t)}$, by 
%
%
%
It\^o's formula  $X$ satisfies the SDE 
\begin{align}\label{AirySDE}
dX(t)=(t-X(t)^2)dt+\tfrac{2}{\sqrt{\beta}}dB(t),
\end{align}
with initial condition $X(0)=c_1/c_0$. The initial condition is $\infty$ if $c_0=0$, $c_1\neq 0$. Note that the diffusion blows up to $-\infty$ at the zeros of $\psi$, and it restarts at $\infty$ instantaneously  whenever this happens. 

The drift in (\ref{AirySDE}) vanishes on the parabola  $x^2=t$, it is positive for $|x|<\sqrt{t}$, and negative for $|x|>\sqrt{t}$. This suggests that the asymptotic behavior of $X(t)$ should be $\sqrt{t}$ (since the branch $x=-\sqrt{t}$ is unstable), as stated in (\ref{Airysqr}). The proposition below proves this statement by providing quantitative bounds on $|X(t)-\sqrt{t}|$. See Figure \ref{fig:asymp_X} for an illustration of the asymptotic behavior of $X$. Note that less precise asymptotic bounds on $X$ were also proved in \cite{DL2} for the study of the small $\beta$ limit.

\begin{proposition}\label{pr:Airy}
Let $\psi, \psi'$ be the strong solution of (\ref{AirySDE_1}) with deterministic initial conditions $\psi(0)=c_0$, $\psi'(0)=c_1$, $(c_0,c_1)\neq (0,0)$. Let $X(t)=\frac{\psi'(t)}{\psi(t)}$. Then there is an a.s.~finite random time $T$ such that 
\begin{align}\label{Airybnd}
|X(t)-\sqrt{t}|\leq t^{-1/4}\ln  t, \quad \text{ for all $t\ge T$}.
\end{align}
\end{proposition}
Our upper bound in (\ref{Airybnd}) is not  optimal. In fact by evaluating the error terms in the proof given below it can be shown  that  $t^{-1/4}\ln  t$  can be replaced with $ t^{-1/4}\sqrt{\ln  t}\; g(t)$ for any positive function $g(t)$ satisfying $\lim_{t\to \infty}g(t) =\infty$.

%

\begin{figure}
\centering
\includegraphics[width = 8cm]{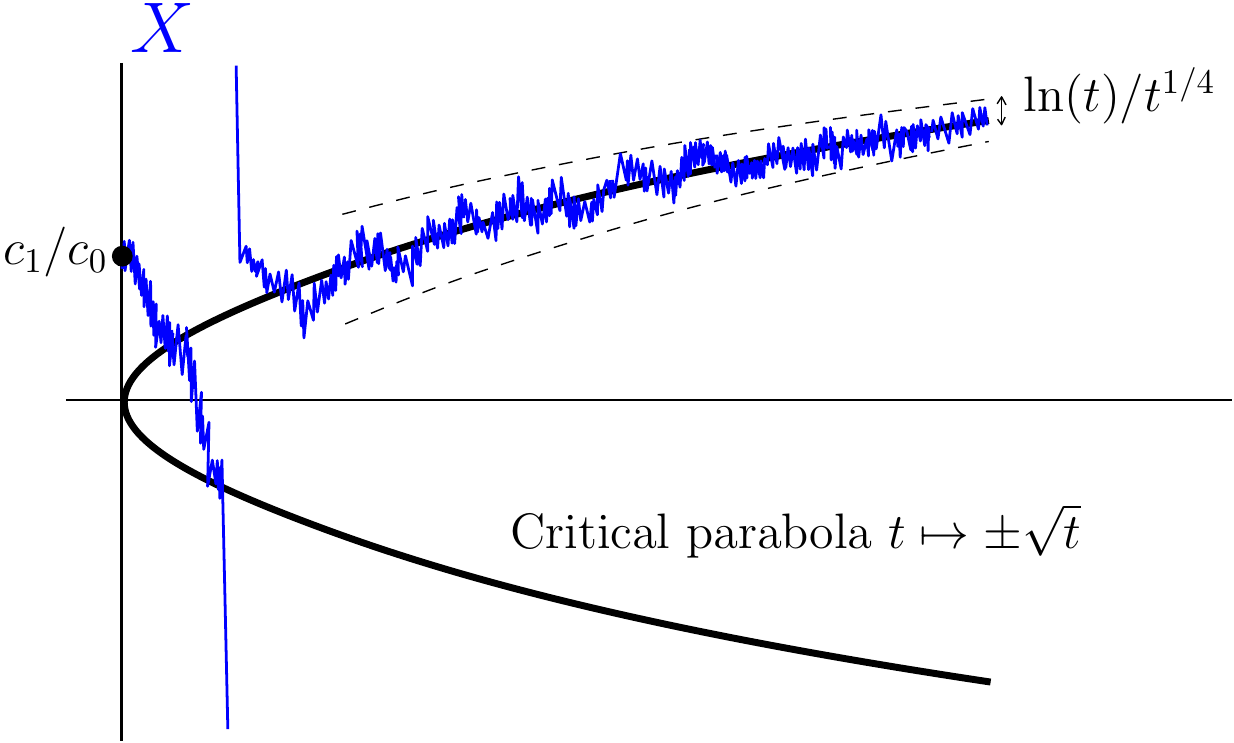}
\caption{Schematic illustration of the asymptotic behavior of the diffusion $X$}\label{fig:asymp_X}
\end{figure}

The proof of Proposition \ref{pr:Airy} relies on the following two technical lemmas, whose proofs are postponed to Section \ref{s:Xlemmas}.

\begin{lemma}\label{l:Airy_close}
Let $X$ be a strong solution of the SDE (\ref{AirySDE}). For a given $s\ge 10$ set
\begin{align}\label{sigma_s}
\sigma_s=\inf \left\{t \geq s\;:\; |X(t)-\sqrt{t}|\le \tfrac12 t^{-1/4}\ln  t\right\}.
\end{align}
Then $\sigma_s$ is a.s.~finite.
\end{lemma}
\begin{lemma}\label{l:Airy_stays}
For a given $t_0>0, x_0\in \R$ consider the solution $X$ of the SDE (\ref{AirySDE}) on $[t_0,\infty)$ with initial condition $X(t_0)=x_0$, and denote by $P_{t_0, x_0}$ its distribution. Then
\begin{align}\label{Airy_close}
    \lim_{t_0\to \infty}\inf_{|x_0-\sqrt{t_0}|\le \frac12 t_0^{-1/4}\ln  t_0} P_{t_0,x_0}\left(|X(t)-\sqrt{t}|\leq t^{-1/4}\ln  t, \text{ for all }t\ge t_0\right)=1.
\end{align}
\end{lemma}

Lemma \ref{l:Airy_close}   shows that for any solution $X$ of the SDE (\ref{AirySDE}) and any $s\ge 10$ the process $X(t)-\sqrt{t}$ will get close enough to 0 after time $s$. 
Lemma \ref{l:Airy_stays} shows that if  $X(t)-\sqrt{t}$ is close to $0$ for a given large $t=t_0$  then with a high probability it will stay close to $0$ for all $t\ge t_0$. 

\begin{proof}[Proof of Proposition \ref{pr:Airy}]
Let $f(t) = t^{-1/4}\ln  t$.  
By Lemma \ref{l:Airy_close}
for any fixed $s\ge 10$ there is an a.s.~finite stopping time $\sigma_s$ with $\sigma_s\ge s$ so that $|X(\sigma_s)-\sqrt{\sigma_s}|\le \frac12 f(\sigma_s)$ with probability one. 
Lemma \ref{l:Airy_stays}  shows that if the diffusion is close to $\sqrt{t}$ then with a high probability it will stay close forever.

 More precisely, for a given $\eps>0$ one can choose $s\ge 10$ so that 
\[
\inf_{\substack{t_0\ge s\\|x_0-\sqrt{t_0}|\le \frac12 f(t_0)}} P_{t_0,x_0}\left(|X(t)-\sqrt{t}|\leq f(t), \text{ for all }t\ge t_0\right)\ge 1-\eps.
\]
The strong Markov property and Lemma \ref{l:Airy_close}  now imply that the inequality  (\ref{Airybnd}) holds with $T=\sigma_{s}$ with probability at least $1-\eps$.  This  shows that the random time
\[
T_0=\inf \left\{s\ge 10: |X(t)-\sqrt{t}|\le  f(t) \text{ for all $t\ge s$} \right\}
\]
is finite with probability at least $1-\eps$, hence it is a.s.~finite. Therefore  (\ref{Airybnd}) holds with probability one with $T=T_0$.
\end{proof}

%
We can now prove Lemma \ref{l:Step1}.

\begin{proof}[Proof of Lemma \ref{l:Step1}]
By Proposition \ref{pr:AirySL} with probability one the operator $\Airyop^{-1}$ is a Hilbert-Schmidt integral operator with kernel $K_{\textsf{A}}$. 
From (\ref{Airysqr}) and the estimate (\ref{psibnd}) it follows that $\psi_\ddd$ has a largest zero (if it has one), hence if $L$ is larger than that,  the linearity of the equation $\Airyop \psi=0$ implies that
\begin{align}\label{def:phiL}
\psi_L(y)=\psi_\infty(y)-\frac{\psi_\infty(L)}{\psi_{\ddd}(L)}\psi_{\ddd}(y).
\end{align}
Hence the truncated operator $K_{\textsf{A}}^{(L)}$ is well-defined in this case. From the definition of $K_{\textsf{A}}^{(L)}$ we get
\begin{align}\label{HS_Airy_1}
\|K_{\textsf{A}}-K_{\textsf{A}}^{(L)}\|_2^2=\iint_{[0,L]^2} \left|K_{\textsf{A}}(x,y)-K_{\textsf{A}}^{(L)}(x,y)\right|^2 dx \, dy+\iint_{\R_+^2\setminus [0,L]^2}\left|K_{\textsf{A}}(x,y)\right|^2 dx \, dy.
\end{align}
By Proposition \ref{pr:AirySL}, with probability one we have $\|K_{\textsf{A}}\|_2^2<\infty$. This implies that the term  $\iint_{\R_+^2\setminus [0,L]^2}\left|K_{\textsf{A}}(x,y)\right|^2 dx \, dy$ converges to 0 a.s.~as $L\to \infty$. In fact, by the arguments described in the proof of Proposition \ref{pr:AirySL} it follows that $\iint_{\R_+^2\setminus [0,L]^2}\left|K_{\textsf{A}}(x,y)\right|^2 dx \, dy$ can be bounded by $C L^{-1/2}$  with a random  constant $C$.

We now estimate the first term on the right hand side of \eqref{HS_Airy_1}. By symmetry we have
\[
\iint_{[0,L]^2} \left|K_{\textsf{A}}(x,y)-K_{\textsf{A}}^{(L)}(x,y)\right|^2 dx \, dy=2 \int_0^L \int_0^y \left|K_{\textsf{A}}(x,y)-K_{\textsf{A}}^{(L)}(x,y)\right|^2 dx \, dy.
\]
From (\ref{def:phiL}), for $L$ large enough, and  $0\le x\le y\le L$,  we get   
\begin{align*}
K_{\textsf{A}}(x,y)-K_{\textsf{A}}^{(L)}(x,y)&=(\psi_\infty(y)-\psi_L(y)) \psi_{\ddd}(x)
=\psi_{\ddd}(x)\psi_{\ddd}(y) \int_L^\infty \psi_{\ddd}(z)^{-2} dz,
\end{align*}
and
\begin{align}\label{KHS}
\int_0^L \int_0^y \left|K_{\textsf{A}}(x,y)-K_{\textsf{A}}^{(L)}(x,y)\right|^2 dx \, dy&=\frac12 \left(\int_0^L \frac{\psi_{\ddd}(x)^2}{\psi_{\ddd}(L)^2} dx\right)^2 \left(\int_L^\infty \frac{\psi_{\ddd}(L)^2}{\psi_{\ddd}(z)^{2}} dz\right)^2.
\end{align}
From the bounds of  \eqref{ineq:prodpsi} we get that the expression in (\ref{KHS}) is bounded by a random constant times $L^{-2}$, and thus it converges to zero a.s.~as $L\to \infty$.  This concludes the proof of Lemma \ref{l:Step1}. 
\end{proof}

\section{Convergence of the truncated operators}\label{s:Step2proof}

Recall the definition of $\tilde \phi_L, \psi_L$ from  Section \ref{s:mainsteps}. 
Lemma \ref{l:Step2} will follow from the following statement:
\begin{lemma}\label{l:SDE_lim}
For any fixed $L>0$ we have $\tilde \phi_{\ddd}\to \psi_{\ddd}$ and $\tilde \phi_L \to \psi_L$ uniformly on $[0,L]$ with probability one as $a\to \infty$. 
\end{lemma}
\begin{proof}[Proof of Lemma \ref{l:Step2}]
From (\ref{KAL}), (\ref{KGL}), and Lemma \ref{l:SDE_lim} it follows that if $L>0$ is fixed then $K_{\textsf{G},2a}^{(L)}(x,y)\to K_{\textsf{A}}^{(L)}(x,y)$ uniformly on $[0,L]^2$ with probability one. From this  Lemma \ref{l:Step2} follows.
\end{proof}

The proof of Lemma \ref{l:SDE_lim}  relies on the following proposition:

%

\begin{proposition}\label{prop:sdelim}
%
Let $B'$ be standard white noise on $\R_+$, and $B$ the corresponding Brownian motion. 
Define $\mathfrak{G}_{\beta,2a}$ using $B_{2a}(x)=a^{-1/3} B(a^{2/3} x)$, and $\Airyop$ with $B'$ as in Theorem \ref{thm:hardtosoft}.  %
Let $\eta_0, \eta_1$ be fixed real numbers. Suppose that the processes $u_a, a\ge 1$ satisfy the following conditions:
\begin{enumerate}
\item[(a)] $\mathfrak{G}_{\beta,2a} u_a=a^2 u_a$, 
\item[(b)] $u_a(0), u_a'(0)$ are deterministic, depend continuously on $a$,  and satisfy 
\[
(a^{2/3} u_a(0), u_a'(0)-a u_a(0))\to (\eta_0, \eta_1)
\]
 as $a\to \infty$.
\end{enumerate}

Let $\hat u_a(x)=a^{2/3} e^{-a^{1/3} x}u_a(a^{-2/3} x)$. Then for any $L>0$ we have $(\hat u_a, \hat u_a') \to (\psi, \psi')$ a.s.~uniformly on $[0,L]$ where $\psi, \psi'$ is the unique solution of $\Airyop \psi=0$ with initial conditions $\psi(0)=\eta_0$, $\psi'(0)=\eta_1$. 
\end{proposition}

\begin{proof}
To ease notation, we drop the dependence on $a$ in  $u_a, \hat u_a$.  By Proposition \ref{pr:hardSL} the process $(u(t), u'(t))$ satisfies the SDE 
\begin{align}\label{sdehard_}
du(x)=u'(x)dx, \quad du'(x)=\tfrac{2}{\sqrt{\beta}} u'(x) dB_{2a}(x)+\left((2a+\tfrac2{\beta})u'(x)-a^2 e^{-x} u(x)\right) dx.
\end{align}
The initial conditions for $\hat u$ are 
\[
\hat u(0)=a^{2/3} u(0), \qquad \hat u'(0)=u'(0)-a u(0),
\]
hence by the conditions of the proposition we see that $(\hat u(0), \hat u'(0))\to (\eta_0, \eta_1)$.  Note that $\hat{u}'(x)=-a^{1/3}\hat{u}(x)+e^{-a^{1/3}x}u'(a^{-2/3}x)$, by It\^o's formula and \eqref{sdehard_} we have that 
\begin{align*}
d \hat u'&=
\tfrac{2}{\sqrt{\beta}}(a^{-1/3}\hat u'+\hat u)dB(x) + \left(a^{2/3}(1-e^{-a^{-2/3}x})\hat u+ \tfrac{2}{\beta}a^{-1/3}\hat u+ \tfrac{2}{\beta}a^{-2/3}\hat u'\right)dx.
\end{align*}
This means that $\hat u, \hat u'$ satisfies
\begin{align}\label{uhat}
d\hat u(x)&=\hat u'(x) dx,\\
d\hat u'(x)&=\hat u(x)(\tfrac{2}{\sqrt{\beta}}d B(x)+x dx)+F_1(\eps, x, \hat u(x),\hat u'(x)) dx+F_2(\eps, x, \hat u(x),\hat u'(x)) dB,\notag
\end{align}
where  $\eps=a^{-1/3}$ and
\begin{align}
F_1(\eps, x, p, q)=(\eps^{-2}(1-e^{-\eps^2x})-x) p+\tfrac2{\beta} \eps p+\tfrac2{\beta} \eps^2 q, \qquad
F_2(\eps, x, p, q)=\tfrac{2}{\sqrt{\beta}}\eps q.
%
%
\end{align}
With a bit of abuse of notation we will use $\hat u_\eps, \hat u'_\eps$ to denote the dependence on $\eps\in (0,1]$.

The functions $F_1, F_2$ can be continuously extended to $\eps=0$  by setting $F_i(0,x,p,q)=0$. Define $(\hat u_0, \hat u_0')$ to be the solution of (\ref{uhat}) with $\eps=0$ and initial conditions $(\eta_0, \eta_1)$.  This  is exactly  the solution $(\psi, \psi')$ of  $\Airyop \psi=0$ and $\psi(0)=\eta_0$, $\psi'(0)=\eta_1$.

Note that for $x\in [0,L]$, $\eps\in [0,1]$ the functions  $F_1$, $F_2$ are globally Lipschitz in $p$ and $q$, and $(\hat u_\eps, \hat u'_\eps), \eps\in [0,1]$ gives a stochastic flow where the deterministic initial conditions are continuous for $\eps\in [0,1]$.  
Standard theory of stochastic flows  (see e.g.~Theorem 37 in Chapter 7 of \cite{Protter}) shows that there is a unique one-parameter family of strong solutions for the SDE (\ref{uhat}) for $\eps\in [0,1]$ which is a.s.~uniformly continuous in $\eps$ for $x\in [0,L]$. But this implies that $(\hat u_\eps, \hat u_\eps') \to (\hat u_0, \hat u_0')$ a.s.~uniformly on $[0,L]$ as $\eps\to 0$, proving the statement of the lemma.
 \end{proof}

\begin{proof}[Proof of Lemma \ref{l:SDE_lim}]
Consider $u_a(x)=\phi_{\ddd}(x)$. These functions satisfy the conditions of Proposition \ref{prop:sdelim} with $\eta_0=0, \eta_1=1$. Thus $\hat u(x)=a^{2/3} e^{-a^{1/3} x} \phi_{\ddd}(a^{-2/3} x)$ converges to $\psi_{\ddd}$ a.s.~uniformly on $[0,L]$ as $a\to \infty$. Then the same is true for 
\[
\tilde \phi_{\ddd}(x)=a^{2/3} m_{2a}^{1/2}(a^{-2/3} x)\phi_{\ddd}(a^{-2/3} x)=\hat u(x) e^{-\frac{a^{-2/3}}{2} x-\frac{a^{-1/3}}{\sqrt{\beta}} B(x)}.
\]
To show the convergence of $\tilde \phi_L$ we first consider $\phi_*$, the solution  of  $\mathfrak{G}_{\beta,2a} \phi_*=a^2 \phi_*$ with initial conditions $\phi_*(0)=a^{-2/3} $, $\phi_*'(0)=a^{1/3}$. Then $v_a(x)=\phi_*(x)$ satisfies  the conditions of Proposition \ref{prop:sdelim} with $\eta_0=1, \eta_1=0$. This means that $\hat v(x)=a^{2/3} e^{-a^{1/3} x}\phi_*(a^{-2/3} x)$ converges uniformly to $\psi_*(x)$ where $\Airyop \psi_*=0$ and $\psi_*(0)=1$, $\psi_*'(0)=0$ (i.e.~the solution with Neumann initial conditions).

%
%

By linearity $\psi_L(x)=\psi_*(x)-\frac{\psi_*(L)}{\psi_{\ddd}(L)}\psi_{\ddd}(x)$.  Note that   $\psi_{\ddd}(L)\neq 0$ with probability one for a fixed $L$, so $\psi_L$ is a.s.~well-defined. This also implies that for a fixed $L$ the random variable $\tilde \phi_{\ddd}(L)$ is not zero if $a$ is larger than a random constant, and in this case $\tilde \phi_L$ is also well-defined.

The function $\psi_L$ satisfies $\Airyop \psi_L=0$ with 
 $\psi_L(0)=1$, $\psi_L(L)=0$. By our previous arguments we have $ \hat v(x)-\frac{\hat v(L)}{\hat u (L)}\hat u(x)\to \psi_L(x)$ a.s.~uniformly for $x\in [0,L]$, as $a\to \infty$. We have
 \begin{align*}
 \hat v(x)-\frac{\hat v(L)}{\hat u (L)}\hat u(x)&=a^{2/3} e^{-a^{1/3} x}\phi_*(a^{-2/3} x)-\frac{a^{2/3} e^{-a^{1/3} L}\phi_*(a^{-2/3} L)}{a^{2/3} e^{-a^{1/3} L}\phi_\ddd(a^{-2/3} L)}a^{2/3} e^{-a^{1/3} x}\phi_\ddd(a^{-2/3} x)\\
 &= a^{2/3} e^{-a^{1/3} x}\left(\phi_*(a^{-2/3} x)-\frac{\phi_*(a^{-2/3} L)}{\phi_\ddd(a^{-2/3} L)}\phi_\ddd(a^{-2/3} x)\right),
 \end{align*}
 and we can check (by plugging in $x=0$ and $x=L$) that 
 \[
  \hat v(x)-\frac{\hat v(L)}{\hat u (L)}\hat u(x)=e^{-a^{-1/3} x} \phi_{a^{-2/3} L}(a^{-2/3} x)=\tilde \phi_L(x)  e^{\frac{a^{-2/3}}{2} x+\frac{a^{-1/3}}{\sqrt{\beta}} B(x)}.
\] 
 But this now implies that $\tilde \phi_L \to \psi_L$ uniformly on $[0,L]$ with probability one, completing the proof. 
\end{proof}

\section{Truncation of the Bessel operator}\label{s:Step3proof}

In order to control $ \|K_{\textsf{G},2a}-K_{\textsf{G},2a}^{(L)}\|_2^2$ and prove Lemma \ref{l:Step3}, we need to understand the asymptotic behavior of $\phi_{\ddd}(t)=\phi_{\ddd}^{(2a)}(t)$ uniformly in $a$. As before, we turn to the Riccati transform $p=p^{(2a)}(t)=\frac{\phi_{\ddd}'(t)}{\phi_{\ddd}(t)}$.
 It\^o's formula together with (\ref{sdehard}) implies that $p(t)$ satisfies the diffusion
\begin{align}
dp(t)=\frac{2}{\sqrt{\beta}}p(t) dB_{2a}(t)+\left((2a+\tfrac{2}{\beta})p(t)-p(t)^2-a^2 e^{-t}\right)dt
\end{align}
with initial condition $p(0)=\infty$. The diffusion could reach $-\infty$ at a finite time, in which case it restarts at $+\infty$ instantaneously. 

Our next proposition describes the behavior of $p$ in the region $[a^{-2/3} L, \infty)$ uniformly in $a$.  In words
the asymptotic behavior of $p$ can be explained as follows: on a microscopic $a^{-2/3}$ time scale the scaled version of $p$ (that is $a^{-2/3}(p(a^{-2/3}t) - a)$) will mimic $\tfrac{\psi_\ddd'(t)}{\psi_\ddd(t)}$ by Proposition \ref{prop:sdelim}, and this behavior can be extended up to a small macroscopic time of order $a^{2/3}$. For large macroscopic times the diffusion $p(t)/a$ will behave like a time-stationary diffusion supported on $\R_+$, which yields logarithmic bounds on $\ln  p(t)-\ln  a$.

For the rest of this section we set $t_0 :=1/8$. Recall that for $a>0$ we have $B_{2a}(t)=a^{-1/3} B(a^{2/3} t)$.

\begin{proposition}[Behavior of the Bessel diffusion]\label{prop:phibnd}
Let $d_1, d_2>0$. 
For a given  $L>0$ and $a_1\ge 1$,  define  $\mathcal{C}_{L,a_1}$ to be the event where the following inequalities hold for all $a\geq a_1$:
\begin{align}
 p^{(2a)}(t)\ge a(1+d_1 \sqrt{t}), \qquad\qquad\qquad &\text{for all }t\in [a^{-2/3}L, t_0], \label{eq:q_bnd1}\\[3pt]
\exp(-a^{-1/6}\ln  t) \le p^{(2a)}(t)/a \le \exp(d_2+a^{-1/6}\ln  t), \qquad &\text{for all }t\ge t_0,\label{eq:q_bnd2}\\[3pt]
\tfrac{2}{\sqrt{\beta}}|B_{2a}(t)-B_{2a}(s)|\leq a^{1/2}(t-s)+a^{-1/6}\ln (a^{2/3}s), \qquad &\text{for all }t\ge s\ge a^{-2/3} L. \label{eq:B2abound}
\end{align}
Then we can choose deterministic constants  $d_1, d_2>0$ so that 
\begin{align} \label{eq:Cla}
\lim\limits_{L\to \infty}\lim\limits_{a_1\to \infty}P\big(\mathcal{C}_{L,a_1}\big)=1\,.
\end{align}
\end{proposition}

\begin{figure}
\centering
\includegraphics[width= 12cm]{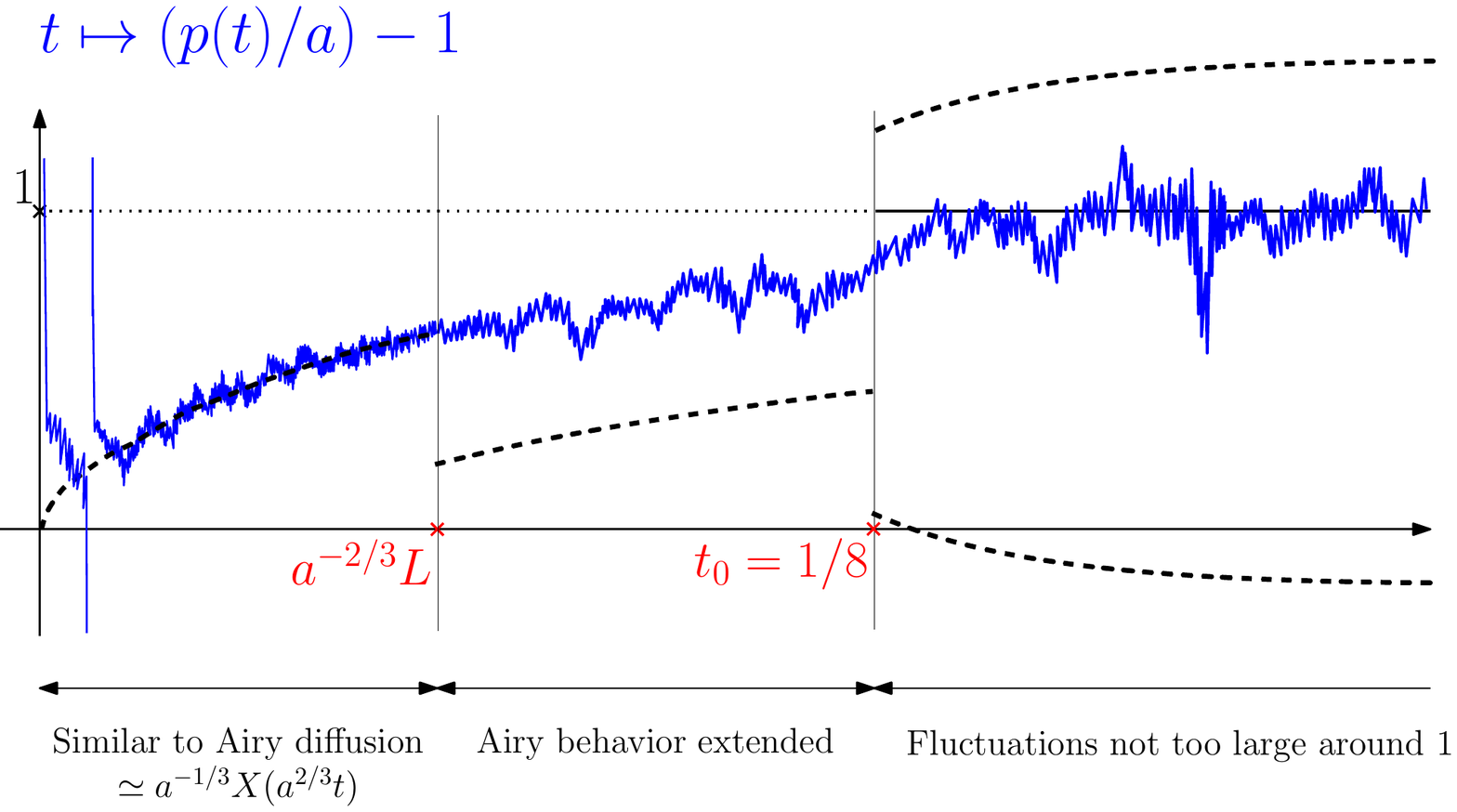}
\caption{Schematic representation of the behavior of the diffusion $t \mapsto (p(t)/a) -1$ 
}\label{Besseldiffusion}
\end{figure}

See Figure \ref{Besseldiffusion} for an schematic illustration of the behavior of the Bessel diffusion. The proof of Proposition \ref{prop:phibnd}  is postponed to Section \ref{s:phi tail}.
Using this proposition we can control the products $\tilde \phi_\ddd(x) \; \tilde \phi_\infty(x)$ and $ \tph_\ddd(y)^{-2}\tph_\ddd(x)^2$ when $y \geq x \geq L$. This will be key to estimate  $ \|K_{\textsf{G},2a}-K_{\textsf{G},2a}^{(L)}\|_2^2$. 

In the rest of this section, we assume $L \ge 10$ and set $c_L=(10L)^{3/2}\vee4(1-e^{-t_0})^{-2}$.

\begin{proposition}\label{prop:Ibnd}
Define
	\begin{align}
	\mathcal{I}(s,t) := -2 \int_{s}^{t}(p(z)-a )dz+\frac{2}{\sqrt{\beta}}(B_{2a}(t)-B_{2a}(s))\,.\label{defmathcalI}
	\end{align}
There are absolute constants $c, c'$ so that for all $a_1\geq c_L$, the following inequalities hold on the event $\mathcal{C}_{L, a_1}$ (as defined in Proposition \ref{prop:phibnd}): for all $a\geq a_1$,
	\begin{align}\label{ineq:mathcalI}
	\mathcal{I}(s,t)\le \begin{cases}
	-c\, a\sqrt{s}(t-s)+c' &\qquad t\geq s, \quad t_0\geq s \geq a^{-2/3}L,\\
	-c \,a(t-s)+5a^{-1/6}\ln  s+c' & \qquad t\geq s\geq t_0.
	\end{cases}
	\end{align}
\end{proposition}
\begin{proof}
We first  prove the case when $t\geq s\geq t_0$ in  (\ref{ineq:mathcalI}). From this point on we will work on the event $\mathcal{C}_{L,a_1}$ with $a_1\ge c_L$, allowing us to assume the inequalities (\ref{eq:q_bnd1})--(\ref{eq:B2abound}). 
	Let us define 
	\[
	q(t) := q^{(2a)}(t) :=\ln  p^{(2a)}(t)-\ln  a.
	\]
	On the event $\mathcal{C}_{L,a_1}$, and for $t \geq t_0$, $q(t)$ is well defined as $p(t) >0$.  By It\^o's formula the process $q$  satisfies the following differential equation:
	\[
	dq(t) = \tfrac{2}{\sqrt{\beta}}dB_{2a}(t) + a(2-e^{q(t)}-e^{-t-q(t)})dt, 
	\]
with the initial condition $q(t_0)=\ln (p(t_0)/a)>0.$
Note that the drift of the diffusion $q$ will be close to $a(2-e^q)$ for large $t$. The  corresponding diffusion 
\[
d\tilde q=\tfrac{2}{\sqrt{\beta}}dB_{2a}(t) + a(2-e^{\tilde q(t)})dt, 
\]
converges to a stationary distribution supported on $\R$ (which can be computed explicitly). This suggests that $q$ behaves like the stationary solution of $\tilde q$, and hence we cannot expect  to get a uniform constant bound on $a(e^{q(t)}-1)=p(t)-a$ in (\ref{defmathcalI}).  Because of this we instead look for a bound on the integral term  in (\ref{defmathcalI}).

We start with the following identity: for all $t \geq s \geq t_0$, we have
	\begin{align}\label{id:qdiffusion}
	a\int_s^t (e^{q(z)}-1)dz &= a(t-s)+\tfrac{2}{\sqrt{\beta}}(B_{2a}(t)-B_{2a}(s))-(q(t)-q(s))-a\int_s^t e^{-q(z)-z}dz\,.
	\end{align}
Using the lower bound from \eqref{eq:q_bnd2} and the fact that $-a^{-1/6}\ln  t\geq -t+t_0$ for all $t\geq t_0$, we get
\begin{align}\label{eq:123bound}
a\int_s^t (e^{q(z)}-1)dz &\geq a(1-e^{-t_0})(t-s)+\tfrac{2}{\sqrt{\beta}}(B_{2a}(t)-B_{2a}(s))-(q(t)-q(s)).
\end{align}
and thus
\[
\mathcal{I}(s,t) \leq -2a(1-e^{-t_0})(t-s)-\tfrac{2}{\sqrt{\beta}} (B_{2a}(t)-B_{2a}(s))+2q(t)-2q(s).
\]
Using the inequality   $\ln  t\leq  \ln  s+t_0^{-1}(t-s)$ for $t\geq s\geq t_0$, the bounds \eqref{eq:q_bnd2}, \eqref{eq:B2abound}, and by our choice of  $c_L$, we get that there exist positive constants  $c_1, c'_1$ such that for all $t \geq s \geq t_0$, we have
\begin{align*}
\mathcal{I}(s,t) \leq - c_1\, a \, (t-s) + 5 a^{-1/6} \ln  s + c'_1\,.
\end{align*}
This completes the proof of (\ref{ineq:mathcalI}) in the case  $t\geq s\geq t_0$.

Let us consider now the case $a^{-2/3}L\leq s< t_0$. From  \eqref{eq:q_bnd1} we have for all $a^{-2/3} L\leq s \leq t\leq t_0$,
\begin{align*}
\int_{s}^{t} (p(z)-a) dz\geq \frac23 \,a\, d_1 (t^{3/2}-s^{3/2}) \geq \frac23\, a\, d_1 \sqrt{s}(t-s) \,.
\end{align*}
Using the lower bound from \eqref{eq:B2abound} we deduce that for all $a^{-2/3} L\leq s \leq t\leq t_0$,
\begin{align*}
\mathcal{I}(s,t) \leq - \frac{4}{3}\, a\, d_1 \sqrt{s}(t-s) + a^{1/2} (t-s)  + a^{-1/6} \ln  (a^{2/3}s)\,.
\end{align*}
As $a^{-2/3}L\leq s\leq t_0 $ and $a\ge a_1\geq c_L$, we  get that there exists a constant $c_I$ such that:
\begin{align*}
\mathcal{I}(s,t) \leq - d_1 \,a \, \sqrt{s}(t-s) + c_I\,.
\end{align*}
For $t \geq t_0 \geq s\geq a^{-2/3}L$, note that $\mathcal{I}(s,t) = \mathcal{I}(s,t_0) + \mathcal{I}(t_0,t)$. Therefore, we get
\begin{align*}
\mathcal{I}(s,t) &\leq - d_1\, a\, \sqrt{s}(t_0-s) + c_I - c_1\, a\,  (t-t_0) + 5 a^{-1/6} \ln  t_0 + c'_1\,,\\
&\leq -c_2\,a\,\sqrt{s}(t-s)+c_I',
\end{align*}
where $c_2= \min\{d_1,c_1t_0^{-1/2}\}$. We choose $c=\min\{c_1,c_2\}$ and $c' = \max\{c'_1,c_I,c'_I\}$ to conclude the proof of (\ref{ineq:mathcalI}). 
\end{proof}

As a consequence of Proposition \ref{prop:phibnd}, we can also show that $a^2$ is not an eigenvalue of $\mathfrak{G}_{\beta,2a}$ if $a$ is large enough. 

\begin{corollary}\label{prop:eigenvalue}
Let $a_1 \geq c_L$. On the event $\mathcal{C}_{L, a_1}$ defined in Proposition \ref{prop:phibnd}, $a^2$ is not an eigenvalue of $\mathfrak{G}_{\beta,2a}$ for all $a\ge a_1$. As a consequence, there exists an a.s. finite random variable $C_{\text{ev}}>0$ such that $a^2$ is not an eigenvalue of $\mathfrak{G}_{\beta,2a}$ on the event $\{a\geq C_{\text{ev}}\}$.
\end{corollary}
\begin{proof}
The value $a^2$ is not an eigenvalue of $\mathfrak{G}_{\beta,2a}$ exactly if the function $\phi_{\ddd}^{(2a)}$ is not in $L^2(\R_+,m_{2a})$. 
On $\mathcal{C}_{L, a_1}$ and for $a \geq a_1$, using the identity (\ref{id:qdiffusion})  and the bound (\ref{eq:123bound}) in the proof of Proposition \ref{prop:Ibnd},  we get
\begin{align*}
a\int_{t_0}^t e^{q(z)}dz 
&\geq a(2-e^{-t_0})(t-t_0)+\tfrac{2}{\sqrt{\beta}}(B_{2a}(t)-B_{2a}(t_0))-q(t)+q(t_0).
\end{align*}
Recall that $ae^{q(t)}=p(t)=\frac{\phi_{\ddd}'(t)}{\phi_{\ddd}(t)}$. Using the above lower bound on the integral of $ae^{q(t)}$, and the bounds (\ref{eq:q_bnd2}) and (\ref{eq:B2abound}), we get
\begin{align*}
	\phi_\ddd(t)^2 m_{2a}(t) &= \phi_\ddd(t_0)^2 \exp\Big(2  \int_{t_0}^t p(z) dz\Big) \exp(- (2a +1) t - \tfrac{2}{\sqrt{\beta}}B_{2a}(t))\\
	&\geq c(t_0)\exp\big(2a(1-e^{-t_0})t-t-a^{1/2}t-2a^{-1/6}\ln  t\big)\;,
\end{align*}
where $c(t_0)$ is an a.s.~finite random constant. Choosing $a\geq a_1\geq c_L\ge (1-e^{-t_0})^{-2}$, we get that  $\int_0^\infty \phi_\ddd(t)^2 m_{2a}(t) dt$ is infinite, proving the statement.

Now set
	\[
	C_{\text{ev}}=1+\inf_{a_1\ge c_L, L\ge 10} a_1\cdot \ind_{\mathcal{C}_{L,a_1}}.
	\]
	If  $a\ge C_{\text{ev}}$ then $a^2$ is not an eigenvalue of $\mathfrak{G}_{\beta,2a}$. 
	By the limit (\ref{eq:Cla}), the random variable $C_{\text{ev}}$ is a.s.~finite, which completes the proof. 
\end{proof}

\begin{proposition}\label{prop:prod} Recall the definition of the event $\mathcal{C}_{L,a_1}$ from Proposition  \ref{prop:phibnd}. On this event $a^2$ is not an eigenvalue of $\mathfrak{G}_{\beta,2a}$ (or $\Besselop$) if $a\ge a_1 \ge c_L$ by Corollary \ref{prop:eigenvalue}, hence $\tilde \phi_\infty$ is well-defined. 
There exist deterministic constants $c_1, c>0$
such that for all $L\geq 10$ and $a_1\ge c_L$, the following inequalities hold on $\mathcal{C}_{L,a_1}$: for all $a\geq a_1$,
	\begin{align}\label{ineq:prodtildephix}
	\tilde \phi_\ddd(x) \tilde \phi_\infty(x)\le \begin{cases}
	c \,x^{-1/2}& \qquad L \le x<a^{2/3} t_0,\\
	c \, a^{-1/3} e^{-a^{-2/3} x/2} &\qquad x\ge a^{2/3} t_0,
	\end{cases}
	\end{align}
and	
	\begin{align}
 \tph_\ddd(y)^{-2}\tph_\ddd(x)^2 \leq \begin{cases} \exp(- c_1 \sqrt{x}(y-x) + c) &\; y \geq x,\quad a^{2/3} t_0 \geq x \geq L \\  
 \exp(- c_1 \,a^{1/3} \, (y-x) + 5 \,a^{-1/6} \ln  x + c) &\; y \geq x \geq a^{2/3} t_0\,.
  \end{cases}
  \label{ineq:phiyphix}
\end{align}
Moreover, under the same conditions, we also get the following inequality for all $y\geq x\geq L$:
\begin{align}\label{ineq:phiyphix2}
\tph_\ddd(y)^{-2}\tph_{\ddd}(x)^2\leq \exp\left( -c_1\sqrt{L}(y-x)+5a^{-1/6}\ln  x+ c\right).
\end{align}
\end{proposition}

\begin{proof}
Recall  the definition of $\tilde \phi_{\ddd}, \tilde \phi_{\infty}$
	from (\ref{phitilde}). 
	On $\mathcal{C}_{L,a_1}$, the diffusion $p(t)$ does not explode on $[a^{-2/3}L,\infty)$, which also implies the largest zero of $\phi_\ddd^{(2a)}$ is smaller than $a^{-2/3}L$. By the Wronskian identity (\ref{wronski}), for all $x \geq L$ we have
	\begin{align}
	{ \tilde\phi_\infty(x)}{\tilde \phi_{\ddd}(x)}
	&=a^{2/3} s(a^{-2/3} x)m_{2a}(a^{-2/3} x)\int_{a^{-2/3}x}^\infty   \phi_{\ddd}(a^{-2/3} x)^2 \phi_{\ddd}(y)^{-2} \frac{s(y)}{s(a^{-2/3} x)} dy\notag \\
	&=e^{-a^{-2/3} x} \int_{x}^\infty  \exp\big(\mathcal{I}(a^{-2/3} x, a^{-2/3} y)\big) dy\,,
	\label{boundtildephix}
	\end{align}
	where
	\begin{align*}
	\mathcal{I}(s,t) := -2 \int_{s}^{t}(p(z)-a )dz+\frac{2}{\sqrt{\beta}}(B_{2a}(t)-B_{2a}(s))\,.
	\end{align*}
	For the product $\tph_d(y)^{-2}\tph_d(x)^2$ for $y\geq x\geq L$, we have
	\begin{align*}
	\tph_\ddd(y)^{-2}\tph_\ddd(x)^2
	&=  \exp\big( a^{-2/3}(y-x)+\mathcal{I}(a^{-2/3}x,a^{-2/3}y)\big)\,.	\end{align*}
For $a_1\geq c_L$, (\ref{ineq:phiyphix}) follows from (\ref{ineq:mathcalI}) directly. Integrating the exponential of (\ref{defmathcalI}) and using the upper bounds (\ref{ineq:mathcalI}), we get (\ref{ineq:prodtildephix}) and the statement of the proposition. The inequality (\ref{ineq:phiyphix2}) follows by comparing the upper bounds in  (\ref{ineq:prodtildephix}) and  (\ref{ineq:phiyphix}).
\end{proof}

We now turn to the proof of Lemma \ref{l:Step3}. We will  use the following identity, that follows from the linearity of the equation $\mathfrak{G}_{\beta,2a}\phi=a^2\phi$: 
\begin{align}
&\tilde\phi_\infty(x)-\tilde\phi_L(x)=m_{2a}^{1/2}(a^{-2/3}x)\frac{\phi_{\infty}(a^{-2/3}L)}{\phi_\ddd(a^{-2/3}L)}\phi_{\ddd}(a^{-2/3}x)=\frac{ \tilde\phi_\infty(L)}{\tilde \phi_{\ddd}(L)}\tph_\ddd(x). \label{identitytildephi}
\end{align}
By Propositions \ref{prop:phibnd} and \ref{prop:Ibnd},  we have that $\tph_\ddd(L)\neq 0$ and $\tph_\infty$ is well-defined for all $a\geq a_1$ on the event $\mathcal{C}_{L,a_1}$.

\begin{proof}[Proof of Lemma \ref{l:Step3}]
For $L\ge 10$ define the event
	\[
	\mathcal{C}_{L}^{(1)} = \Big\{\psi_\ddd(K)^{-2}\int_0^K\psi_\ddd(x)^2dx\leq 2K^{-1/2},\quad \mbox{for all } K\geq L\Big\}\cap\{\psi_\ddd(t)>0,\quad \forall t\geq L\}\,.
	\]
	The family of events $\mathcal{C}_L^{(1)}, L\ge 10$ is non-decreasing in $L$ and $\lim_{L\to\infty}\mathbb{P}(\mathcal{C}_L^{(1)})=1$, by Proposition \ref{pr:Airy}. Define the events
	\[
	\mathcal{C}_{L,a_1}^{(2)} = \mathcal{C}_{L,a_1} \cap \mathcal{C}_{L}^{(1)}\cap \Big\{\tph_\ddd^{(2a)}(L)^{-2}\int_0^L \tph_\ddd^{(2a)}(x)^2dx\leq 3L^{-1/2},\quad \forall a\geq a_1\Big\}\,.
	\]
    The family $\mathcal{C}_{L,a_1}^{(2)}$ is non-decreasing in  $a_1$ for fixed $L$ and the events $\cup_{a_1} \mathcal{C}_{L,a_1}^{(2)}$ are non-decreasing in $L$. By the uniform convergence of $(\tph_\ddd,\tph_\ddd')\to(\psi,\psi')$ on $[0,L]$, we have 
    \[
    \lim_{L\to\infty}\lim_{a_1\to\infty}\mathbb{P}(\mathcal{C}_{L,a_1}^{(2)})=1.
    \] 
 We now prove inequalities on the event $\mathcal{C}_{L,a_1}^{(2)}$ for all $a_1\ge c_L$. In the following, $c'$ is a  constant that may change from line to line. We start with the following identity: 
	\[
	\|K_{\textsf{G},2a}-K_{\textsf{G},2a}^{(L)}\|_2^2=\iint_{[0,L]^2} \left|K_{\textsf{G},2a}(x,y)-K_{\textsf{G},2a}^{(L)}(x,y)\right|^2 dx \, dy+\iint_{\R_+^2\setminus [0,L]^2}\left|K_{\textsf{G},2a}(x,y)\right|^2 dx \, dy.
	\]
	On $[0,L]^2$ we have
	\begin{align*}
	\iint_{[0,L]^2} \left|K_{\textsf{G},2a}(x,y)-K_{\textsf{G},2a}^{(L)}(x,y)\right|^2 dx \, dy &=2\int_0^L\int_0^y \tilde \phi_{\ddd}(x)^2 ( \tilde\phi_\infty(y)-\tilde\phi_L(y))^2 dx dy\,,\\
	&=\left(\tilde \phi_{\ddd}(L)^{-2}\int_0^L \tilde \phi_{\ddd}(x)^2 dx \right)^2 { \tilde\phi_\infty(L)^2}{\tilde \phi_{\ddd}(L)^2}\,,\\
	&\leq (3L^{-1/2})^2(cL^{-1/2})^2\,,
	\end{align*}
	using identity \eqref{identitytildephi} for the second line and  the bound (\ref{ineq:prodtildephix})  for $x = L$ for the third line. Thus this term is bounded by $c'L^{-2}$ uniformly in $a$.

    We further split the region  $\R_+^2\setminus [0,L]^2$ into the union of  $\mathcal{R}_1=[L,\infty)\times [0,L]\cup [0,L]\times [L,\infty)$ and $\mathcal{R}_2 = [L,\infty)^2$.
    On $\mathcal{R}_1$ we have:
	\begin{align*}
	\iint_{\mathcal{R}_1} \left|K_{\textsf{G},2a}(x,y)\right|^2 dx \, dy
	&=\Big(2 \tilde \phi_{\ddd}(L)^{-2} \int_0^L  \tilde \phi_{\ddd}(x)^2 dx \Big) \tph_{\ddd}(L)^{2} \int_L^\infty \tilde\phi_\infty(y)^2 dy\,.
	\end{align*}
    The first term $2\, \tilde \phi_{\ddd}(L)^{-2} \int_0^L  \tilde \phi_{\ddd}(x)^2 dx$ is bounded from above by $6\, L^{-1/2}$. For the second term, we split the integral, and apply Proposition \ref{prop:prod} to get the following upper bound:
	\begin{align*}
	\tph_{\ddd}(L)^{2}& \int_L^\infty \tilde\phi_\infty(y)^2 dy\\
	& = \int_{L}^{a^{2/3}t_0}\tph_\infty(y)^2\tph_\ddd(y)^2 \tph_\ddd(y)^{-2}\tph_\ddd(L)^2 dy + \int_{a^{2/3}t_0}^\infty\tph_\infty(y)^2\tph_\ddd(y)^2 \tph_\ddd(y)^{-2}\tph_\ddd(L)^2 dy\\
	& \leq \int_{L}^{a^{2/3}t_0} c^2y^{-1}e^{-c_1 \sqrt{L}(y-L)+c}dy +\int_{a^{2/3}t_0}^\infty c^2a^{-2/3}e^{-a^{-2/3}y}e^{-c_1\sqrt{L}(y-L)+c}dy\\
	& \leq c' (L^{-3/2}+L^{-1/2}a^{-2/3}).
	\end{align*}

    At last, on $\mathcal{R}_2$ we have
	\begin{align*}
	\iint_{\mathcal{R}_2} \left|K_{\textsf{G},2a}(x,y)\right|^2 dx dy 
	&= 2\int_{L}^{a^{2/3}t_0}\int_L^y \tph_{\ddd}(x)^2 \tph_\infty(y)^2 dx dy + 2\int_{a^{2/3}t_0}^\infty\int_L^y\tph_{\ddd}(x)^2 \tph_\infty(y)^2 dx dy.
	\end{align*}
	We use (\ref{ineq:prodtildephix}) and (\ref{ineq:phiyphix}) to bound the first integral,
	\begin{align*}
	\int_{L}^{a^{2/3}t_0}\int_L^y \tph_{\ddd}(x)^2 \tph_\infty(y)^2 dx dy&= \int_L^{a^{2/3}t_0}\tph_\ddd(y)^2\tph_\infty(y)^2\int_L^y \tph_\ddd(y)^{-2}\tph_\ddd(x)^2 dx dy\\
	&\le \int_L^{a^{2/3}t_0}c^2 y^{-1}\int_L^y e^{-c_1\sqrt{x}(y - x) + c} dxdy\\
	&\le \int_L^{a^{2/3}t_0}c' y^{-3/2} dy\\
	&\le c' L^{-1/2}\,.
	\end{align*}
		For the second integral, we use (\ref{ineq:prodtildephix}) and (\ref{ineq:phiyphix2}), 
	\begin{align*}
	\int_{a^{2/3}t_0}^\infty \int_L^{y}  \tph_\infty(y)^2 \tph_{\ddd}(x)^2  dx dy &= \int_{a^{2/3}t_0}^\infty \tph_\infty(y)^2\tph_\ddd(y)^2 \int_L^y \tph_\ddd(y)^{-2}\tph_{\ddd}(x)^2  dx dy\\
	&\leq  \int_{a^{2/3}t_0}^\infty c^2a^{-2/3}e^{-a^{-2/3}y}\int_{ L}^y e^{-c_1 \sqrt{L}(y-x)+5a^{-1/6}\ln  y+ c}dx dy\\
	&\leq \int_{a^{2/3}t_0}^\infty c'L^{-1/2}a^{-2/3}e^{-a^{-2/3}y+5a^{-1/6}\ln  y} dy\\
	&\leq c'L^{-1/2}\,.
	\end{align*}
	
	
Recall that the family of events $\mathcal{C}_{L,a_1}^{(2)}$ is non-decreasing in $a_1$ for fixed $L$, and the events $\mathcal{C}_{L}^{(2)}:=\cup_{a_1}\mathcal{C}_{L,a_1}^{(2)}$ satisfy $\mathcal{C}_{L}^{(2)}\uparrow \Omega$ as $L\to\infty$ with $P(\Omega)=1$. On the event $\Omega$ we have
\[
	\lim\limits_{L\to \infty} \limsup\limits_{a\to \infty} \|K_{\textsf{G},2a}-K_{\textsf{G},2a}^{(L)}\|_2^2=0,
\]
which completes the proof.
\end{proof}

\begin{remark}\label{rmk:convergence rate}
Note that our estimates give an upper bound of the order $O(L^{-1/2})$ on the squared Hilbert-Schmidt norm difference of $K_{\textsf{G},2a}$ and $K_{\textsf{G},2a}^{(L)}$. A bound of the same order was shown on the  truncation error for  $K_{\textsf{A}}$. 

By choosing $L=L_a$ to be dependent on $a$ with $L_a\to \infty$ at some rate, one could potentially obtain a bound on the rate of convergence in (\ref{HSlim}). This would require the extension of the result of Lemma \ref{l:SDE_lim} to increasing intervals $[0,L_a]$. We do not explore this path in this paper, but we want to present a hand-waving argument to show that our methods are not expected to give better than logarithmic convergence.  


In the proof of Proposition \ref{prop:sdelim}, we  viewed the process $(\hat u,\hat{u}')$ as a stochastic flow depending on two variables $\eps=a^{-1/3}$ and $x$. It is reasonable to expect that if the statement of  Lemma \ref{l:SDE_lim} holds on the interval $[0,L_a]$ then $\sup_{x\leq L_a}|\hat u_\eps(x)-\hat u_0(x)|$ should vanish as $a\to \infty$. This quantity should be of the same order as $\eps \sup_{x\leq L_a} |v(x)|$ where $v(x)=\partial_\eps \hat u_\eps(x)|_{\eps=0}$.  One can check that $v$ satisfies the stochastic differential equation,
\begin{align*}
dv=v'dx,\quad dv'=v(\tfrac{2}{\sqrt{\beta}}dB+xdx) + \tfrac{2}{\beta}\hat u_0(x)dx + \tfrac{2}{\sqrt{\beta}}\hat u_0'(x)dB
\end{align*}
with initial values $v(0)=0$ and $v'(0)=0$. If we assume that $v'$ grows at least as fast as the contribution of the $\tfrac{2}{\beta}\hat u_0(x)dx$ term then we would get that $v$ grows at least as fast as $e^{\tfrac12 x^{3/2}}$. This would lead to the requirement $a^{-1/3} e^{\tfrac12 L_a^{3/2}}\to 0$, and $L_a\ll(\ln a)^{2/3}$. Hence the speed of convergence could not be faster than $(\ln a)^{-1/3}$.
\end{remark}

\section{Bounds on the soft and hard edge diffusions}
\subsection{Asymptotic properties of the soft edge diffusion $\psi_{\ddd}$}
\label{s:Xlemmas}

This section contains the proofs of Lemma \ref{l:Airy_close} and  \ref{l:Airy_stays},  which were used for the asymptotic analysis of the diffusion $X$ in (\ref{AirySDE}). In this section  we set $f(t)=t^{-1/4}\ln  t$.

\begin{proof}[Proof of Lemma \ref{l:Airy_close}]
We will prove that  
\begin{align}\label{region1}
\lim_{t_0\to \infty} P\left(|X(t)-\sqrt{t}|\le \tfrac12 f(t) \text{ for some $t\in [t_0,t_0 +  \tfrac{1}{\sqrt{t_0}}\ln ^3(t_0)]$}\right)=1.
\end{align}
This means that with higher and higher probability we will hit the region $|X(t)-\sqrt{t}|\le \frac12 f(t)$ within a small time interval, which implies that $\sigma_s<\infty$ with probability one. 

To prove \eqref{region1} we consider $X$ with initial condition $X(t_0)=x_0$ with $t_0\ge 10$, $x_0\in \R$, and give a bound on the probability in (\ref{region1}) in each of the following cases (see Figure \ref{differentcases_RT}):
\begin{align*}
    \text{Case I:\qquad}& \qquad\qquad\qquad x_0>\sqrt{t_0}+f(t_0)/2\\
    \text{Case II:\qquad}& \qquad\qquad\qquad x_0<-\sqrt{t_0}-f(t_0)\\
    \text{Case III:\qquad}& -\sqrt{t_0}+f(t_0)<x_0<\sqrt{t_0}-\frac12 f(t_0)\\
    \text{Case IV:\qquad}& -\sqrt{t_0}-f(t_0)\le x_0\le -\sqrt{t_0}+f(t_0).
\end{align*}
In each one of these cases we will compare the diffusion to a time-homogeneous version of itself.  Then 
in Cases I-III we use the idea that as long as we control the maximal value of the Brownian motion $B$, the diffusion will stay close to the deterministic path solving the ODE $x(t)' = t- x(t)^2$ which is what we get if we remove the noise from the SDE of $X$. In Case IV we will use explicit computations about hitting times of diffusions.

\begin{figure}
\centering
\includegraphics[width = 9cm]{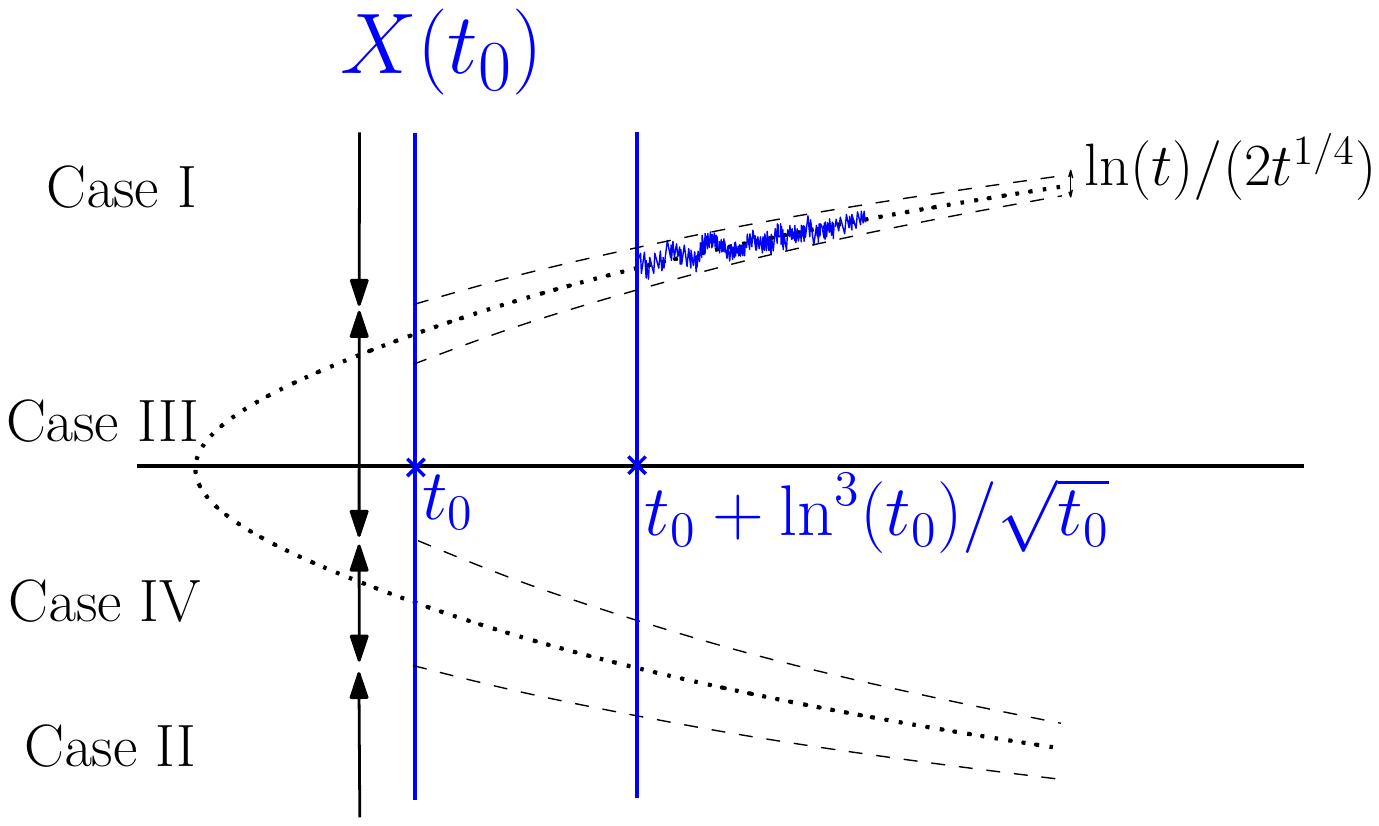}
\caption{Representation of the four different cases for the position of $X(t_0)$}\label{differentcases_RT}
\end{figure}

Let $g(x)=x+\frac{1}{\sqrt{x}}\ln (x)$. We consider Case I, when $x_0 > \sqrt{t_0} + f(t_0)/2$. We set  $t_1=g(t_0)$ and assume that $t_0$ is large enough. 
Let the time-homogeneous diffusion $X_{+}$ on $[t_0, t_1]$ be given by the strong solution of 
\[
dX_{+}(t) = (t_1 - X_{+}(t)^2)dt + \tfrac{2}{\sqrt{\beta}} dB(t), \qquad X_{+}(t_0)=+\infty.
\]
Comparing the drifts of $X_+$ and $X$ we see that on the event $\{X(t)>\sqrt{t}, t\in [t_0,t_1]\}$ we have $X_+(t)\ge X(t)$ for $t\in [t_0,t_1]$. 

The process $Z(t):= X_{+}(t) - \tfrac{2}{\sqrt{\beta}}\, \tilde B(t)$ with $\tilde B(t)=B(t)-B(t_0)$ satisfies the ODE
\[
Z'(t) =t_1 - Z(t)^2(1+ \tfrac{\tfrac{2}{\sqrt{\beta}} \tilde B(t)}{Z_t})^2, \qquad Z(t_0)=\infty
\]
for all time $t\ge t_0$ smaller than the first hitting time of $0$  for $Z$. We set 
\[
M := \tfrac{1}{10}f(t_0)=\tfrac{1}{10}t_0^{-1/4}\ln  t_0,
\]
and introduce the event 
\[
\mathcal{A}=\mathcal{A}_{t_0}:= \{\sup_{t \in [t_0,t_1]}\;|B(t) - B(t_0)| \leq \tfrac{\sqrt{\beta}}{2} M\}.
\]
Note that 
\[
P(\mathcal{A})= P\left(\sup_{s\in [0,1]} |B(s)|\le \tfrac{\sqrt{\beta}}{20}\sqrt{\ln  t_0}\right)
\]
which shows that $P(\mathcal{A}_{t_0})\to 1$ as $t_0\to \infty$. 

On  the event $\mathcal{A}$, if $Z(s)=\sqrt{t_0}$ for an $s\in [t_0,t_1]$ then this would imply 
\[
X(s)\le \sqrt{t_0}+M\le \sqrt{s}+f(s)/2.
\]
On $\widetilde{\mathcal{A}}=\mathcal{A}\cap \{Z(t)>\sqrt{t_0}, t\in [t_0, t_1]\}$, $Z$ is bounded from above by the deterministic solution of 
\[
F'(t) = t_1 - F(t)^2(1-2M/\sqrt{t_0}), \qquad F(t_0)=\infty,
\]
which is given by 
\begin{align*}
F(t)=\sqrt{t_1/D}\coth(\sqrt{t_1D}(t-t_0)), \qquad D=1-2M/\sqrt{t_0}.
\end{align*}
Using Taylor-expansion, we get that for   $t_0$  large enough we have  $F(t_1)  \leq \sqrt{t_0} + 2  M$ which implies that on $\widetilde{\mathcal{A}}$ we must have $X(t_1)\le \sqrt{t_0}+3M\le \sqrt{t_1}+\frac12 f(t_1)$. 
This shows that 
\[
\mathcal A\subset \{|X(t)-\sqrt{t}|\le \tfrac12 f(t) \text{ for some } t\in [t_0,t_1]\},
\]
which implies
\[
\lim_{t_0\to \infty}\inf_{x_0>\sqrt{t_0}+\frac12f(t_0)}P_{x_0, t_0}\left(|X(t)-\sqrt{t}|\le \tfrac12 f(t) \text{ for some $t\in [t_0,t_0 +  \tfrac{1}{\sqrt{t_0}}\ln ^3(t_0)]$}\right)=1.
\]

Next we consider the case $x_0<-\sqrt{t_0} - f(t_0)$ (this is Case II).
Similar arguments used as in Case I show that for $t_0$ large enough   $X$ explodes to $-\infty$ before time $t_1=g(t_0)$ on the event $\mathcal{A}$. Since $X$ restarts at $+\infty$ at the explosion, we are back in Case I, and by the arguments presented there 
 we get that $|X(t)-\sqrt{t}|\le \frac12 f(t)$ must hold  before time $g(t_1)$ with high probability.   Since $g(t_1)\le t_0+\ln ^3 t_0/\sqrt{t_0}$ for $t_0$ large, we get 
 \[
\lim_{t_0\to \infty}\inf_{x_0<-\sqrt{t_0}-\frac12f(t_0)}P_{x_0, t_0}\left(|X(t)-\sqrt{t}|\le \tfrac12 f(t) \text{ for some $t\in [t_0,t_0 +  \tfrac{1}{\sqrt{t_0}}\ln ^3(t_0)]$}\right)=1.
\]

Now consider  Case III, when $x_0 \in (-\sqrt{t_0} + f(t_0), \sqrt{t_0} - f(t_0)/2)$. We show that   $X$
 reaches $\sqrt{t_0} - f(t_0)/2$ before time $t_1$ with probability going to $1$. For this we can just assume  that $x_0 = -\sqrt{t_0} + f(t_0)$, since the other cases stochastically  dominate this one by a simple coupling.  Let us examine again $Z = X - \tfrac{2}{\sqrt{\beta}} \tilde B$. The process $Z(t)$ satisfies the ODE
 \[
 Z'(t)=t-(Z(t)+\tfrac{2}{\sqrt{\beta}}\tilde B(t))^2, \quad Z(t_0) = -\sqrt{t_0}+f(t_0). 
 \]
On the event $\mathcal{A}$, the process $Z$ is increasing when $-\sqrt{t}+M\leq Z(t)\leq \sqrt{t}-M$, in particular $Z'(t_0)>0$. Before $Z$ hits $\sqrt{t_0}$, we can bound $Z$ from below by $G(t)$ where
\[
G'(t) =  (\sqrt{t_0}-\tfrac32M)^2-G^2(t),\quad G(t_0) = -\sqrt{t_0}+f(t_0). 
\]
Solving the above initial value problem, we get $G(t) = (\sqrt{t_0}-\tfrac32 M)\tanh((\sqrt{t_0}-\tfrac32M)(t-t_0)+c)$ where $c<0$ is chosen such that $G(t_0) = -\sqrt{t_0} + f(t_0)$. Here $c\sim -\frac38 \ln  t_0$ if $t_0$ is large. Using Taylor-expansion again, we get $G(t_1)\geq \sqrt{t_0}-2M$ which implies that $X(t)\ge \sqrt{t}-f(t)/2$ somewhere in $[t_0,t_1]$.

For the last case IV when $x_0 \in [-\sqrt{t_0} - f(t_0),-\sqrt{t_0} + f(t_0)]$, denote by $\tau$ the exit time of $X(t)$ from the interval $[q^-,q^+]:=[-\sqrt{t_1}-f(t_1),-\sqrt{t_1}+f(t_1)]$. We use the time-homogeneous diffusion $\tilde X(t)$ satisfying the SDE
\[
d\tilde X(t) = (t_0 - \tilde X(t)^2)dt + \tfrac{2}{\sqrt{\beta}} dB(t),\quad \tilde X(t_0)=x_0.
\] 
Let us denote by $\tilde \tau$ the first exit time for $\tilde X$ after time $t_0$ from $(q^-,q^+)$. By the Cameron-Martin-Girsanov formula, the Radon-Nikodym derivative of $X$ with respect to $\tilde X$ on the time interval $[t_0,t_1]$ can be expressed as $e^{G(\tilde X)}$ where
\begin{align*}
G(\tilde X) = \frac{1}{(\tfrac{2}{\sqrt{\beta}})^2} \Big( \int_{t_0}^{t_1} (\tilde X(t_1) - \tilde X(t)) dt -\frac{(t_1-t_0)^3}{6} - \int_{t_0}^{t_1} t \, (t_0 - \tilde X(t)^2) dt\Big)\,.
\end{align*}
On the event $\{\tilde X(t) \in [q^-,q^+], t\in [t_0,t_1]\}$  one can bound $G(\tilde X)$ by a constant.
This means that $P(\tau>t_1)$ can be bounded by a constant times $P(\tilde \tau>t_1)$.

We can explicitly compute $E[\tilde \tau]$ in terms of the scale function and speed measure of $\tilde X$. The scale function $sc$ and speed measure $sp$ for $\tilde X(t)$ are given by
\[
sc(x) = \int_{-\infty}^x \exp(-2t_0y+\frac23 y^3)dy,\quad sp(dx) = \frac{2}{sc'(x)}dx.
\]
From this we can express the first moment of $\tilde \tau$ as
\begin{align*}
E[\tilde \tau-t_0]=&\int_{q^-}^{x_0} \frac{(sc(y)-sc(q^-))(sc(q^+)-sc(x_0))}{sc(q^+)-sc(q^-)}sp(dy)\\& + \int_{x_0}^{q^+}\frac{(sc(x_0)-sc(q^-))(sc(q^+)-sc(y))}{sc(q^+)-sc(q^-)}sp(dy).
\end{align*}
(See for example Theorem VII.3.6 \cite{RY}.)
By analyzing the above integrals as $t_0\to\infty$, one can bound $E[\tilde \tau-t_0]$ by $c\tfrac{\ln \ln  t_0}{\sqrt{t_0}}$ with an absolute constant $c$ for all $t_0$ large enough and all $x_0\in[-\sqrt{t_0}-f(t_0),-\sqrt{t_0}+f(t_0)]$. (We refer to Lemma 5.7. of \cite{DL} for additional details for this argument.) By Markov's inequality, we get
\begin{align*}
P_{x_0,t_0}[\tau>t_1] = E[\exp(G(\tilde X))\mathbf{1}_{\{\tilde \tau > t_1\}}] \leq c' \frac{\ln \ln  t_0}{(t_1-t_0)\sqrt{t_0}}=c'\frac{\ln \ln  t_0}{\ln  t_0} \,,
\end{align*}
with an absolute constant $c'$. 
Therefore $X$ exits the region $(q^-,q^+)$ before time $t_1$ with probability tending to $1$ as $t_0\to \infty$. Once $X(t)$ exits this region, we get to Case II or III, and 
repeating the arguments there we can show that
 \[
\lim_{t_0\to \infty}\inf_{x_0:|x_0+\sqrt{t_0}|\leq f(t_0)}P_{x_0, t_0}\left(|X(t)-\sqrt{t}|\le \tfrac12 f(t) \text{ for some $t\in [t_0,t_0 +  \tfrac{1}{\sqrt{t_0}}\ln ^3(t_0)]$}\right)=1.
\]
This completes the proof of (\ref{region1}) and hence the statement of the lemma.
\end{proof}

\begin{proof}[Proof of Lemma \ref{l:Airy_stays}]
Introduce $Y(t):=X(t)-\sqrt{t}$, then $Y(t)$ satisfies the stochastic differential equation
\[
dY(t) = (-Y(t)^2-2\sqrt{t}Y(t)-\frac{1}{2\sqrt{t}})dt + \tfrac{2}{\sqrt{\beta}} dB(t),
\]
with initial condition $y_0=x_0-\sqrt{t_0}$.

With the same driven noise $dB$, we define two families of  diffusions $Y_1(t)=Y_{1}^{y_0,t_0}(t)$, $Y_2(t)=Y_2^{y_0,t_0}(t)$ on $[t_0,  \infty)$ with initial condition $y_0$ as follows:
\begin{align*}
dY_1(t) &= -2\sqrt{t}Y_1(t)  dt + \tfrac{2}{\sqrt{\beta}} dB(t),\quad Y_1(t_0)=y_0,\\
dY_2(t) &= (-2\sqrt{t}Y_2(t) - 2 f(t)^2)dt + \tfrac{2}{\sqrt{\beta}} dB(t), \quad Y_2(t_0)=y_0.
\end{align*}
By comparing the drift terms in $Y, Y_1, Y_2$ we see that if for a given $t_0$ we start $Y_1, Y_2$ from $y_0=Y(t_0)$ at time $t_0$ then 
 the coupling $Y_2(t)\leq Y(t)\leq Y_1(t)$  holds for all $t\ge t_0$ on the event 
 \begin{align}
 \mathcal{D}_{t_0,y_0}:=\{-f(t)\le Y_2(t), Y_1(t)\le f(t) \text{ for all $t\geq t_0$}\}.
 \end{align}
 Consequently, this shows that
\begin{align}\label{cpl1}
\mathcal{D}_{t_0, y_0}\subset \{|Y(t)|\leq f(t),\forall t\geq t_0\},
\end{align}
and thus it is enough to prove
\begin{align}\label{probD}
    \lim_{t_0\to \infty} \inf_{|y_0|\le \frac12 f(t_0)} P(\mathcal{D}_{t_0,y_0})=1.
\end{align}

Using the integrating factor trick, both $Y_1$ and $Y_2$ can be  solved explicitly: 
\begin{align*}
Y_1(t) &=\exp(-\tfrac{4}{3}(t^{3/2}-t_0^{3/2}))y_0 + \tfrac{2}{\sqrt{\beta}} e^{-\frac{4}{3}t^{3/2}}\int_{t_0}^t e^{\frac{4}{3}s^{3/2}}d B_s,\\
Y_2(t) &=\exp(-\tfrac{4}{3}(t^{3/2}-t_0^{3/2}))y_0 - 2 \,e^{-\frac{4}{3}t^{3/2}}\int_{t_0}^t f^2(s) e^{\frac{4}{3}s^{3/2}}ds + \tfrac{2}{\sqrt{\beta}} e^{-\frac{4}{3}t^{3/2}}\int_{t_0}^t e^{\frac{4}{3}s^{3/2}}d B_s.
\end{align*}
Let $\xi(t) = \int_1^t e^{\frac{8}{3}s^{3/2}}ds$ . There exists a Brownian motion $W$ such that we have the following distributional identity:
\[
\left(\int_1^t e^{\frac{4}{3}s^{3/2}}dB_s,\,  t\ge 1 \right){\,\buildrel d\over =\,}  (W(\xi(t)), \, t\ge 1) .
\]
By the Law of Iterated Logarithm, there exist finite random constant $C$ such that
\[
|W (u)|\leq C \sqrt{u\ln \ln  u},\quad \text{ for all $u\geq 20$}.
\]
Note that $\xi(t)\leq \tfrac12 e^{\frac{8}{3}t^{3/2}}t^{-1/2}$ for all $t\geq 1$. We may assume $t_0\ge \max(10, \xi^{-1}(20))$, then for $t\geq t_0$ we get
\begin{align*}
Y_1(t) &\leq \tfrac12  e^{-\frac{4}{3}(t^{3/2}-t_0^{3/2})}f(t_0) + \tfrac{2}{\sqrt{\beta}}\,C\,e^{-\frac{4}{3}t^{3/2}}(\sqrt{\xi (t)\ln \ln  \xi(t)}+\sqrt{\xi(t_0)\ln \ln  \xi(t_0)})\\
&\leq e^{-\frac{4}{3}(t^{3/2}-t_0^{3/2})}(\tfrac{1}2 f(t_0) + \tfrac{2}{\sqrt{\beta}} C t_0^{-1/4}\sqrt{\ln  t_0}) + \,\tfrac{2}{\sqrt{\beta}} C t^{-1/4}\sqrt{\ln  t}.
\end{align*}
Integration by parts yields the bound
\[
\int_{t_0}^t f(s)^2 e^{\frac{4}{3}s^{3/2}}ds \leq  \frac1{ \sqrt{t}} \, f(t)^2 e^{\frac{4}{3}t^{3/2}}.
\]
Therefore, we obtain that
\begin{align*}		
	Y_2(t)&\geq - e^{-\frac43 (t^{3/2} - t_0^{3/2})} (\tfrac{1}{2} f(t_0) + \tfrac{2}{\sqrt{\beta}} C t_0^{-1/4} \sqrt{\ln  t_0}) - 2t^{-1/2} f(t)^2 -\tfrac{2}{\sqrt{\beta}} C t^{-1/4} \sqrt{\ln  t}\,. 
\end{align*}
For a large enough deterministic $c_0$, we have $- f(t)\leq Y_2(t)\leq Y_1(t)\leq f(t)$ for all $t\geq t_0\ge c_0$  on the event $\{C<\tfrac{\sqrt{\beta}}{20}\sqrt{\ln  t_0}\}$. 
Hence if $t_0\ge c_0$ then 
\[
\inf_{|y_0|\le \frac12 f(t_0)} P(	\mathcal{D}_{y_0, t_0})\ge P(C<\tfrac{\sqrt{\beta}}{20}\sqrt{\ln  t_0})
\]
which completes the proof of (\ref{probD}).
\end{proof}

\subsection{Bounds for the hard edge diffusion}\label{s:phi tail}


We start this section with a lemma controlling the fluctuations of  Brownian motion. Although the bounds in the  lemma are not optimal they are sufficient for our purposes.

\begin{lemma}\label{BM fluctuation}
Let $B$ be a standard Brownian motion. Then there is a random finite positive $C$ so that a.s.~we have the following inequality:
\begin{align}\label{BM_0}
|B(s+h)-B(s)|\le C \sqrt{h \ln (2+\tfrac{s}{h}+|\ln  h|)},\qquad  \text{for all }h >0, \;s >0.
\end{align}
 This implies in particular the following simple bounds:
\begin{align}\label{BM_modulus}
|B(s+h)-B(s)|\le C_1 (h+\ln  s), \qquad \text{for all } \quad h>0, \;s\ge 10,
\end{align}
with a random constant $C_1$.
%
%
\end{lemma}

\begin{proof}
First set $h=2^n$, $s=m 2^n$, for $n \in \Z$ and $m \in \N$. We have
\begin{align*}
P(\max_{x\le h}|B(s)-B(s+x)|\ge 4\cdot  2^{n/2} \sqrt{\ln (2+|n|+m)})&\le 2 P(|B(1)|\ge 4  \sqrt{\ln (2+|n|+m)})\\
&\le  2e^{-8 \ln (2+|n|+m)}=\frac{2}{(2+|n|+m)^{8}},
\end{align*}
which is summable for $n\in \Z, m\in \N$. Hence by the Borel-Cantelli Lemma, there is a random $\tilde C$ so that 
\begin{align}\label{BM_maxh}
\max_{x\le h}|B(s)-B(s+x)|\le  \tilde C  \sqrt{h} \sqrt{\ln (2+|\ln  h|+\tfrac{s}{h})}
\end{align}
for all $s=m 2^n$, $h=2^n$. For general $s>0,h>0$, there exist $n\in\Z,m\in\N$ such that $2^n<h\leq 2^{n+1}$ and $m2^n< s\leq  (m+1)2^n$. Using (\ref{BM_maxh}) and the triangle inequality, we get
\[
|B(s+h)-B(s)|\leq 
 8\tilde C\sqrt{h\ln(2+|\ln h|+\tfrac{s}{h})},
\]
which proves the first part of the lemma with $C=8\tilde C$.

For $s\ge 10$ we have
\begin{align*}
	 \sqrt{h \ln (2+\tfrac{s}{h}+|\ln  h|)} 
	 \leq \sqrt{h\ln (2+\tfrac{1}{h}+|\ln  h|)+h\ln (1+s)}.
\end{align*}
For $h\geq \ln  s$, we have
\[
\ln (2+\tfrac{1}{h}+|\ln  h|)<h,\quad \ln (1+s)< \ln  (2s)\leq 2h,
\]
which implies $ \sqrt{h \ln (2+\tfrac{s}{h}+|\ln  h|)} \le 2 h$ in this case.

Now assume $h<\ln  s$. We have   $h\ln  (2+\frac2h+\ln  h)\le 2$ for  $h\in [0,1]$, which yields $h\ln  (2+\frac1h+|\ln  h|) \leq 2 \ln(s) \ln \ln(s)$ for $h<\ln  s$, $s\ge 10$. We also have $h \ln (1+s) \leq (3/2) (\ln  s)^2$ under the same conditions, which 
yields $ \sqrt{h \ln (2+\tfrac{s}{h}+|\ln  h|)} \le 2\ln s$. The bound (\ref{BM_modulus}) now follows from (\ref{BM_0}).
\end{proof}

The next lemma  gives estimates on the  diffusion $p^{(2a)}(t)$ at time $t=a^{-2/3} L$ using the convergence result of Proposition \ref{prop:sdelim}.
\begin{lemma}\label{lem:controltimeL}
For all positive $L$ and $a_1$, let  $\mathcal{A}^{(1)}_{L,a_1}$ be the event that 
\begin{align*}
a\big(1+\tfrac{4}{5}a^{-1/3}\sqrt{L}\big)\leq p(a^{-2/3}L)\leq a\big(1+\tfrac{6}{5}a^{-1/3}\sqrt{L}\big),   \text{ for all $a \geq a_1$}.
\end{align*}
Then $\lim_{L \to \infty} \lim_{a_1 \to \infty} P(\mathcal{A}^{(1)}_{L,a_1}) = 1$.
\end{lemma}

\begin{proof}
The uniform convergence of Proposition \ref{prop:sdelim} implies that almost surely,
\begin{equation}\label{eq:unif_conv}
p(a^{-2/3} L) a^{-2/3}-a^{1/3}\to X(L)\,,\quad \mbox{as }a \to \infty\,.
\end{equation}
Indeed
\begin{align*}
(a^{2/3} \phi(a^{-2/3} t) e^{-a^{1/3} t}, \phi'(a^{-2/3} t) e^{-a^{1/3}t}-a \phi(a^{-2/3} t) e^{-a^{1/3} t})\to  (\psi(t), \psi'(t))\,,
\end{align*}
uniformly on $[0,L]$ and $p(t)=\phi'(t)/\phi(t)$ and  $X(t) = \psi'(t)/\psi(t)$. 

Fix $L$ large and define the event:
\[
\mathcal{A}_L:=\{\tfrac{9}{10}\sqrt{t} \leq X(t)\leq \tfrac{11}{10}\sqrt{t}, \quad\forall t\geq L\}.
\]
Note that the family $\mathcal{A}_L$ is non-decreasing in $L$. From Proposition \ref{pr:Airy} it follows that $\lim_{L\to\infty} \mathbb{P}(\mathcal{A}_L)=1$. For all $L$ and $a_1$, define
\[
\mathcal{A}_{L,a_1} = \mathcal{A}_L\cap\{a(1+\tfrac{4}{5}a^{-1/3}\sqrt{L})\leq p(a^{-2/3}L)\leq a(1+\tfrac{6}{5}a^{-1/3}\sqrt{L}),\quad \forall a\geq a_1\}.
\]
By (\ref{eq:unif_conv}) and the condition $\tfrac{9}{10}\sqrt{L}\leq X(L)\leq \tfrac{11}{10}\sqrt{L}$ on $\mathcal{A}_L$, we have $\mathcal{A}_{L,a_1}\uparrow\mathcal{A}_L$ as $a_1\to\infty$ which concludes the proof.
\end{proof}

Let us introduce $q = q^{(2a)} =\ln  p^{(2a)}-\ln  a$. By Lemma \ref{lem:controltimeL}, the diffusion $q$ is well-defined at time $a^{-2/3} L$ on the event $\mathcal{A}^{(1)}_{L,a_1}$.  By It\^o's formula, for  $t \geq a^{-2/3} L$ we have
\begin{align}
dq(t)=\tfrac{2}{\sqrt{\beta}}dB_{2a}(t)+a(2 -e^{q(t)} - e^{-t-q(t)}) dt\,. \label{SDE q}
\end{align}
The diffusion $q$ blows-up when $p$ reaches $0$, so $q$ may not be well-defined on the whole interval $[a^{-2/3} L,+\infty)$. 

The next proposition controls the growth of $q$ from small times starting at $a^{-2/3} L$ until a positive deterministic time. In this time-interval, $q$ is small and therefore $p$ is close to $a(1+q)$. Analyzing the drift of the $q$ diffusion  for small  $t$ and $q$, we see that  one can compare the behavior of $q$ with the diffusion  $X$ defined in (\ref{AirySDE}).  This allows us to bound  $q$ with constant multiples of the  square root function with large probability.

\begin{proposition}\label{Airy approx}
Fix $t_0:=1/8$. For all positive $L$ and $a_1$ with $a_1^{-2/3}L\leq t_0$, we define $\mathcal{A}^{(2)}_{L,a_1}$ to be the event  that
\begin{align}
\tfrac{2}{5} \sqrt{t} \leq q^{(2a)}(t) \leq \tfrac{7}{5} \sqrt{t}, \quad \forall t\in [a^{-2/3}L,t_0]\,\label{ineq_q_tillt0}\quad \text{ for all $a\geq a_1$.}
\end{align}
Then $\lim_{L \to \infty} \lim_{a_1 \to \infty} P\big(\mathcal{A}^{(2)}_{L,a_1}\big) = 1$. 
\end{proposition}

Note that the inequality \eqref{ineq_q_tillt0} implies
\[
p^{(2a)}(t)\geq a(1+\tfrac{2}{5}\sqrt{t}), \quad \forall t\in [a^{-2/3}L,t_0]\,\quad \text{ for all $a\geq a_1$.}
\] 

\begin{proof}
If $a_1>(8 L)^{3/2}$ then on the event $\mathcal{A}^{(1)}_{L,a_1}$ of Lemma \ref{lem:controltimeL}, we have
\[
\tfrac{3}{5} \sqrt{L}\le  a^{1/3}q(a^{-2/3} L)\le \tfrac{6}{5} \sqrt{L}, \qquad \text{ for all $a \geq a_1$.}
\] 
For  $0\leq q \le 1/2, t\le t_0$ we have the following inequalities:
\[
-q^2+t\ge 2 -e^q - e^{-t-q}=2-e^q-e^{-q}+e^{-q}(1-e^{-t})\ge -2 q^2+\tfrac12 t\,.
\]
Let $q_1=q_{1}^{(2a)}$ and $q_2=q_{2}^{(2a)}$ be the diffusions on $[a^{-2/3}L,t_0]$ so that
\[
d q_1(t)=\tfrac{2}{\sqrt{\beta}}  dB_{2a}(t)+a(\tfrac12  t -2 q_1(t)^2) dt, \quad d q_2(t)=\tfrac{2}{\sqrt{\beta}}  dB_{2a}(t)+a(t-q_2(t)^2) dt\,,
\]
with $q_1(a^{-2/3}L)=q_2(a^{-2/3}L)=q(a^{-2/3}L)$. Then the coupling $\{q_1(t)\leq q(t)\leq q_2(t)\}$ holds on the event $\{0\leq q_1(t)\leq q_2(t)\leq 1/2,\forall t\in[a^{-2/3}L,t_0]\}$.

Recall that  $B_{2a}(t)=a^{-1/3} B(a^{2/3} t)$. Setting $y_1(t)= 2\,a^{1/3} q_1(a^{-2/3} t)$ and  $y_2(t)= a^{1/3} q_2(a^{-2/3} t)$, we get
\[
dy_1(t)=\tfrac{4}{\sqrt{\beta}} dB(t)+(t- y_1(t)^2) dt, \quad dy_2(t)=\tfrac{2}{\sqrt{\beta}} dB(t)+( t-y_2(t)^2) dt, 
\]
with $\frac{6}{5}\sqrt{L}\le y_1(L)\le \frac{12}{5}\sqrt{L}$ and $\frac35 \sqrt{L}\le y_2(L)\le \frac{6}{5}\sqrt{L}$. Thanks to Proposition \ref{pr:Airy}, we know that the event
\begin{align}\label{event1234}
\left\{\forall t \geq L,\quad y_1(t) \in [\tfrac{4}{5}\sqrt{t}, \tfrac{13}{5}\sqrt{t}],\quad y_2(t) \in [\tfrac12 \sqrt{t}, \tfrac75 \sqrt{t}]\right\}
\end{align}
has probability going to $1$ when $L \to \infty$. On the event (\ref{event1234}) we have
\[
0\leq \tfrac{2}{5}\sqrt{t}\leq q_1(t)\leq q(t)\leq q_2(t)\leq \tfrac75 \sqrt{t}\leq \tfrac12, \quad \forall t\in[a^{-2/3}L,t_0],
\]
implying that $p(t)\geq a(1+\frac{2}{5}\sqrt{t})$ on $[a^{-2/3}L, t_0]$.
\end{proof}

Next we estimate the growth of $q(t)$ in the time interval $t\in[t_0,\infty)$. As we will see, $q$ will have a different behavior for large times:  it oscillates near the value $\ln  2$ with possibly making  large excursions away from this value. We will prove  bounds on those fluctuations using a comparison with a non-exploding, stationary version of the diffusion $q$.

\begin{proposition}\label{Log bound} Recall the definition of $\mathcal{A}^{(2)}_{L,a_1}$ from Proposition \ref{Airy approx}. Define 
\[
\mathcal{A}^{(3)}_{L,a_1}  = \mathcal{A}_{L,a_1}^{(2)}\cap \{-a^{-1/6}\ln  t \leq q^{(2a)}(t)\le c+ a^{-1/6} \ln  t, \forall t\geq t_0, \forall a\geq a_1\}\,.
\]
Then, there exists a constant $c>0$ such that
$\lim_{L \to \infty} \lim_{a_1 \to \infty} P\big(\mathcal{A}^{(3)}_{L,a_1}\big) = 1$. 
\end{proposition}

\begin{proof}
	For each $a$, we bound $q(t)$ using two stationary diffusions $q_1(t)=q_1^{(2a)}(t)$ and $q_2(t)=q_2^{(2a)}(t)$, and we show that the growth of $q_1,q_2$ is at most logarithmic with  a large probability.  Let $q_1$ and $q_2$ be the following diffusions:
	\[
	dq_1(t)=\tfrac{2}{\sqrt{\beta}}dB_{2a}(t) + a(c_1-e^{q_1(t)})dt,\quad dq_2(t)=\tfrac{2}{\sqrt{\beta}}dB_{2a}(t) + a(c_2-e^{q_2(t)})dt,
	\]
	with $c_1=2-e^{-t_0}, c_2=2$, and $q_1(t_0)=q_2(t_0)=q(t_0)$. Comparing the drift terms of $q, q_1, q_2$ we see that 
	 the event $\{q_1(t)\geq -t + t_0,\forall t\geq t_0\}$ implies the event  $\{q_1(t)\leq q(t)\leq q_2(t), \forall t\geq t_0\}$.
	
	Notice that the SDEs for $q_{i}$ for $i=1,2$ can be solved. We get that for $t\geq t_0$, $i=1,2$,
	\begin{align*}
	\exp(-q_i(t)) &= \exp(-q_i(t_0)) \exp\Big(a\,c_i(t_0-t)+\tfrac{2}{\sqrt{\beta}}(B_{2a}(t_0)-B_{2a}(t))\Big)\\
	&\quad +a\int_{t_0}^t \exp\Big(a\,c_i(s-t)+\tfrac{2}{\sqrt{\beta}}(B_{2a}(s)-B_{2a}(t))\Big)ds.
	\end{align*}
	Recall that $B_{2a}(t)=a^{-1/3}B(a^{2/3}t)$. Applying the bound (\ref{BM_modulus}) of Lemma \ref{BM fluctuation} on the event $\{C_1 < a_1^{1/6}\}$ for the Brownian motion $B$, we have the following inequality for $x\geq a^{-2/3}L$, $L \geq 10$ and for all $a\ge a_1$:
	\begin{align}
 	\tfrac{2}{\sqrt{\beta}}|B_{2a}(x+h)-B_{2a}(x)| 
	\leq C_1a^{-1/3}(a^{2/3}h+\ln  (a^{2/3}x))
	\leq a^{1/2}h+a^{-1/6}\ln  (a^{2/3}x).\label{pfeqB2abound}
	\end{align}
Note that this is exactly inequality \eqref{eq:B2abound} of Proposition \ref{prop:phibnd}.
	
	Moreover, on $\mathcal{A}^{(2)}_{L,a_1}$, for $a \geq a_1$, we have $\exp(q(t_0))\geq \exp(2\sqrt{t_0}/5)> c_1$. We get that there is an absolute constant $c_3>0$ so that 	for all $a \geq a_1\ge c_3$ we have
	\begin{align*}
	e^{-q_1(t)}
	&\leq \exp\big(-q(t_0)+(ac_1-a^{1/2})(t_0-t)+a^{-1/6}\ln (a^{2/3}t_0)\big)\\
	&\quad +\exp\big(a^{-1/6}\ln (a^{2/3}t)\big)(c_1-a^{-1/2})^{-1}\Big(1-\exp\big((ac_1-a^{1/2})(t_0-t)\big)\Big)\\	
	&\leq e^{a^{-1/6}\ln (a^{2/3}t)}\left((c_1-a^{-1/2})^{-1}+e^{(ac_1-a^{1/2})(t_0-t)}(e^{-q(t_0)}-(c_1-a^{-1/2})^{-1}) \right)\\
	&\leq t^{a^{-1/6}}\,.
	\end{align*}
	We conclude that for all $a \geq a_1\ge c_3$ we have
	\[
	q_1(t)\geq -a^{-1/6}\ln  t \geq -t+t_0, \quad\forall t\geq t_0,
	\]
	which also implies that the coupling $q_2(t)\geq q(t)\geq q_1(t)$ holds on $\{C_1 < a_1^{1/6}\} \cap \mathcal{A}^{(2)}_{L,a_1}$.
		
	For the upper bound, first note that $\exp(q(t_0))<e^{1/2}<c_2=2$ on $\mathcal{A}^{(2)}_{L,a_1}$. Then there is an absolute constant $c_4>0$, so that 
	for all  $a \geq a_1\ge c_4$ and $t \geq t_0$, we have
	\begin{align*}
	e^{-q_2(t)}&\geq e^{-a^{-1/6}\ln (a^{2/3}t)}\left( (c_2+a^{-1/2})^{-1}+e^{(ac_2+a^{1/2})(t_0-t)}(e^{-q(t_0)}-(c_2+a^{-1/2})^{-1})\right)\\
	&\geq e^{-a^{-1/6}\ln  a^{2/3}-q(t_0)}t^{-a^{-1/6}}\,.
	\end{align*}
	Therefore, we deduce
	\[
	-a^{-1/6}\ln  t\leq q(t)\leq a^{-1/6}\ln  t + 1, \quad \forall t\geq t_0
	\]
	on the event  $\{C_1 < a_1^{1/6}\} \cap \mathcal{A}^{(2)}_{L,a_1}$ 	
	for all  $a \geq a_1\ge c_5$ with a fixed $c_5>0$, which completes the proof of  the proposition. \end{proof}

Now we are ready to complete the proof of  Proposition \ref{prop:phibnd}.

\begin{proof}[Proof of  Proposition \ref{prop:phibnd}]
The statement follows from Propositions  \ref{Airy approx} and \ref{Log bound}, and the inequality \eqref{pfeqB2abound}.
\end{proof}

\begin{remark}\label{rmk:qfluctuaction}
A more careful analysis of the diffusion $\phi^{(2a)}_\ddd$ (using ideas described in the proofs of  Lemma \ref{Airy_close} and Lemma \ref{Airy approx}) can provide a logarithmic bound on the diffusion $q$ for a fixed $a>0$. 
%
More precisely, it can be shown that for a fixed $a>1/2$ with probability one the diffusion $q$ satisfies $|q(t)| \leq  \frac{2(32)^2}{\beta \,a} \ln t$ for all large $t$. In particular, this result implies that $\phi_\ddd:= \phi^{(2a)}_\ddd$ is a.s. not in $L^2(\R_+,m_{2a})$ for $a > 1/2$ thanks to the identities (\ref{id:qdiffusion}) and 
	\begin{align*}
	\phi_\ddd(t)^2 m_{2a}(t) = \phi_\ddd(t_0)^2 \exp(2 \,a \int_{t_0}^t e^{q(s)} ds) \exp(- (2a +1) t - \frac{2}{\sqrt{\beta}}B_{2a}(t)).
	\end{align*}
\end{remark}

\noindent\textbf{Acknowledgements.} 
The authors thank Cyril Labb\'e, Brian Rider and B\'alint Vir\'ag for valuable discussions. LD and BV thank the hospitality of Centre International de Rencontres Math\'ematiques in Marseille where part of this work was originated. This research was partially supported by the ANR-16-CE93-0003 (LD) and the NSF award DMS-1712551 (BV).

\def\cprime{$'$}

\Addresses
\end{document}